\documentclass[reqno]{amsart}
\usepackage{amssymb}
\usepackage{amsfonts}
\usepackage{amsmath}
\usepackage{amsthm, amsmath, amssymb, enumerate}
\usepackage{amssymb, enumitem}
\usepackage{bbm}
\usepackage{amscd}
\usepackage[all]{xy}
\usepackage[hidelinks]{hyperref}
\usepackage{amsmath}
\usepackage{amssymb}
\usepackage{amsthm}
\usepackage[sort&compress,numbers]{natbib}
\usepackage[left=2cm,right=2cm,bottom=3cm,top=3cm]{geometry}
\usepackage{tikz}
\usepackage{tikz-cd}
\usepackage{wrapfig}

\newtheorem{theorem}{Theorem}[section]
\newtheorem{proposition}[theorem]{Proposition}
\newtheorem{lemma}[theorem]{Lemma}
\newtheorem{corollary}[theorem]{Corollary}
\theoremstyle{definition}
\newtheorem*{claim*}{Claim}
\newtheorem{definition}[theorem]{Definition}

\newtheorem{remark}[theorem]{Remark}

\AtBeginDocument{   \def\MR#1{}}
\input{tcilatex}

\begin{document}
\title{(Looking For) The Heart of Abelian Polish Groups}
\author{Martino Lupini}
\address{Dipartimento di Matematica, Universit\`{a} di Bologna, Piazza di
Porta S. Donato, 5, 40126 Bologna,\ Italy}
\email{martino.lupini@unibo.it}
\urladdr{http://www.lupini.org/}
\thanks{The author was partially supported by the Marsden Fund Fast-Start
Grant VUW1816, the Rutherford Discovery Fellowship VUW2002 from the Royal
Society of New Zealand, and the Starting Grant ``Definable Algebraic
Topology'' from the European Research Council.}
\date{\today }

\begin{abstract}
We prove that the category $\mathcal{M}$ of abelian groups with a Polish
cover introduced in collaboration with Bergfalk and Panagiotopoulos is the
left heart of (the derived category of) the quasi-abelian category $\mathcal{%
A}$ of abelian Polish groups in the sense of Beilinson--Bernstein--Deligne
and Schneiders. Thus, $\mathcal{M}$ is an abelian category containing $%
\mathcal{A}$ as a full subcategory such that the inclusion functor $\mathcal{%
A}\rightarrow \mathcal{M}$ is exact and finitely continuous. Furthermore, $%
\mathcal{M}$ is uniquely characterized up to equivalence by the following
universal property: for every abelian category $\mathcal{B}$, a functor $%
\mathcal{A}\rightarrow \mathcal{B}$ is exact and finitely continuous if and
only if it extends to an exact and finitely continuous functor $\mathcal{M}%
\rightarrow \mathcal{B}$. In particular, this provides a description of the
left heart of $\mathcal{A}$ as a concrete category.

We provide similar descriptions of the left heart of a number of categories
of algebraic structures endowed with a topology, including: non-Archimedean
abelian Polish groups; locally compact abelian Polish groups; totally
disconnected locally compact abelian Polish groups; Polish $R$-modules, for
a given Polish group or Polish ring $R$; and separable Banach spaces and
separable Fr\'{e}chet spaces over a separable complete non-Archimedean
valued field.
\end{abstract}

\maketitle

\section{Introduction}

The category of abelian groups with a Polish cover and Borel-definable group
homomorphisms was recently introduced in collaboration with Bergfalk and
Panagiotopoulos \cite{bergfalk_definable_2020,bergfalk_definable_2022}. In
this work, we showed that several classical invariants from homological
algebra and algebraic topology, including \textrm{Ext} of countable groups,
Steenrod homology of compact metrizable spaces, and \v{C}ech cohomology of
locally compact metrizable spaces can be seen as functors to the category of
abelian groups with a Polish cover. These provide \emph{definable }%
refinements of such invariants that are \emph{finer}, \emph{richer}, and 
\emph{more rigid }than the purely algebraic ones.

In this paper we prove that the category $\mathcal{M}$ of abelian groups
with a Polish cover and Borel-definable group homomorphisms is an abelian
category. The category $\mathcal{A}$ of abelian Polish groups is a full
subcategory of $\mathcal{M}$, such that the inclusion functor $\mathcal{A}%
\rightarrow \mathcal{M}$ is finitely continuous and exact. Furthermore, $%
\mathcal{M}$ is characterized up to equivalence by the following universal
property: a functor from the category $\mathcal{A}$ of abelian Polish groups
to an abelian category is finitely continuous and exact if and only if it is
isomorphic to a functor that extends to a finitely continuous exact functor
on $\mathcal{M}$, in which case such an extension is unique up to
isomorphism. In other words, $\mathcal{M}$ together with the inclusion $%
\mathcal{A}\rightarrow \mathcal{M}$ is a \emph{universal arrow }\cite[%
Section III.1]{mac_lane_categories_1998} from $\mathcal{A}$ to the forgetful
functor from the category of abelian categories and finitely continuous
exact functors (identified up to isomorphism) to the category of
quasi-abelian categories and finitely continuous exact functors (identified
up to isomorphism). This universal property also identifies $\mathcal{M}$ as
the \emph{left heart }\textrm{LH}$\left( \mathcal{A}\right) $ of the
quasi-abelian category $\mathcal{A}$ (where \textquotedblleft left heart of $%
\mathcal{A}$\textquotedblright\ stands for \textquotedblleft the heart of
the derived category of $\mathcal{A}$ with respect to its canonical left
truncation structure\textquotedblright ) as constructed in \cite%
{schneiders_quasi-abelian_1999,rump_almost_2001}; see also \cite%
{beilinson_faisceaux_1982,buhler_algebraic_2011,rump_abelian_2020}.

The core of the proof consists in showing that $\mathcal{M}$ is indeed an
abelian category, which is far from obvious. This is obtained by means of
tools from descriptive set theory, including a selection theorem for Borel
relations of Kechris and Macdonald \cite{kechris_borel_2016}, and a
dichotomy theorem for coset equivalence relations of Solecki \cite%
{solecki_coset_2009}. After it is established that $\mathcal{M}$ is abelian,
this category can be recognized as the left heart of $\mathcal{A}$ by means
of the characterization of the left heart provided in \cite[Proposition
1.2.36]{schneiders_quasi-abelian_1999}.

It is natural to consider the \emph{left }heart of $\mathcal{A}$, rather
than the dual notion of right heart. Indeed, $\mathrm{LH}\left( \mathcal{A}%
\right) $ has the property that the inclusion $\mathcal{A}\rightarrow 
\mathrm{LH}\left( \mathcal{A}\right) $ preserves finite limits, and in
particular maps monomorphisms to monomorphisms (but does not map
epimorphisms to epimorphisms, in general). This is desirable, since $%
\mathcal{A}$ already has the \textquotedblleft right\textquotedblright\
monomorphisms, which are the injective continuous group homomorphisms.
However, $\mathcal{A}$ has in some sense \textquotedblleft too
many\textquotedblright\ epimorphisms, being these the continuous group
homomorphisms \emph{with dense image}. This is corrected in \textrm{LH}$%
\left( \mathcal{A}\right) $, where the epimorphisms are precisely the \emph{%
surjective} Borel-definable homomorphisms. The fact that the forgetful
functor from $\mathcal{A}$ to the category $\mathbf{Ab}$ of discrete abelian
groups is not finitely cocontinuous can be seen as a manifestation of the
fact that $\mathcal{A}$ has \textquotedblleft too many\textquotedblright\
epimorphisms. In particular, the forgetful functor $\mathcal{A}\rightarrow 
\mathbf{Ab}$ does not extend to a functor on the right heart of $\mathcal{A}$%
. However, being exact and finitely continuous, it extends to a forgetful
functor $\mathrm{LH}\left( \mathcal{A}\right) \rightarrow \mathbf{Ab}$. This
shows that the objects of $\mathrm{LH}\left( \mathcal{A}\right) $ can be
regarded as groups with additional structure, but the same cannot be said
for the objects of the right heart.

We provide similar descriptions of the left heart of several naturally
occurring quasi-abelian categories, including:

\begin{itemize}
\item the category of non-Archimedean abelian Polish groups;

\item the category of locally compact abelian Polish groups;

\item the category of totally disconnected locally compact Polish groups;

\item the category of Polish $R$-modules, for a given Polish group or Polish
ring $R$;

\item the categories of separable Fr\'{e}chet spaces and separable Banach
spaces over a Polish non-Archimedean valued field $K$.
\end{itemize}

The left heart of a quasi-abelian category is also described in \cite%
{schneiders_quasi-abelian_1999,beilinson_faisceaux_1982} as a category of
formal quotients; see also the work of Waelbroeck and Vasilescu on spaces of
formal quotients of Banach, Fr\'{e}chet, or bornological spaces \cite%
{buhler_algebraic_2011,waelbroeck_bornological_2005,wegner_heart_2017,waelbroeck_quotient_1982,waelbroeck_quotient_1989,waelbroeck_category_1986,waelbroeck_category_1993,vasilescu_spectral_1987,vasilescu_spectral_1989,vasilescu_spectral_1988}%
. However, in that context the morphisms are defined abstractly by formally
inverting certain arrows. In this context, we identify the morphisms as a
concrete collection of group homomorphisms that satisfy the natural
requirement of being Borel-definable. This provides for each of the
quasi-abelian categories mentioned above a description of the left heart as
a \emph{concrete category}.

These results make available to the study of abelian Polish groups and
groups with a Polish cover tools from category theory and homological
algebra. Furthermore, they provide the foundation stone for the study of
homological functors on abelian Polish groups and their derived functors.
Homological algebra in the context of locally compact abelian Polish groups
has been studied in \cite%
{moskowitz_homological_1967,hoffmann_homological_2007,fulp_extensions_1971,prosmans_derived_1999}%
. The concrete description of the left heart of abelian Polish groups
provided in this paper has already found a number of applications. These
include the characterization of the injective and projective objects in the
left heart of locally compact Polish abelian groups \cite%
{bergfalk_applications_2023}, and the calculation of the potential Borel
complexity of the classification problem for extensions of countable abelian
groups \cite{lupini_classification_2022,lupini_classification_2022-1}.

In this paper we begin with recalling in Section \ref{Section:category} the
notions of abelian category, quasi-abelian category, and the left heart of a
quasi-abelian category. We then present in Section \ref{Section:polish-cover}
the notions of Borel-definable set and Borel-definable group from \cite%
{bergfalk_definable_2022,lupini_definable_2020,lupini_definable_2020-1}, and
the more restrictive notion of \emph{group with a Polish cover }from \cite%
{bergfalk_definable_2020}. In Section \ref{Subsection:subgroups} we
introduce the notion of Polishable subgroup of a group with a Polish cover.
The main result here is that images and preimages of Polishable subgroups of
abelian groups with a Polish cover are Polishable; see Proposition \ref%
{Proposition:preimage}. In Section \ref{Section:complexity} we reformulate
in this context some results concerning the Borel complexity of Polishable
subgroups from \cite{lupini_complexity_2022}. In Section \ref%
{Section:solecki} we describe a canonical chain of Polishable subgroups of a
given abelian group with a Polish cover, which we call \emph{Solecki
subgroups}. These were originally defined by Solecki in \cite%
{solecki_polish_1999}, and have also been considered in \cite%
{farah_borel_2006,solecki_coset_2009}. Ulm subgroups of abelian groups with
a Polish cover are introduced in Section \ref{Section:Ulm}. Section \ref%
{Subsection:modules} explains how all the results obtained up to that point
apply more generally to \emph{Polish }$R$-\emph{spaces }for a fixed Polish
group or Polish ring $R$, and in particular to Polish $K$-vector spaces for
a Polish field $K$.

In Section \ref{Section:better-lifts} we show that in certain circumstances
a Borel-definable $R$-homomorphism has a lift that is well-behaved with
respect to the algebraic structure. Finally, in Section \ref{Subsection:LH-A}
we prove the characterization of the left heart of the category of Polish $R$%
-modules (Theorem \ref{Theorem:MR-abelian}), and in Section \ref%
{Subsection:LH-A} a more general result describing the left heart of a
strictly full quasi-abelian subcategory of the category of Polish $R$%
-modules (Theorem \ref{Theorem:left-heart-B}). The latter is applied in
Section \ref{Subsection:examples} to describe the left heart of a number of
categories of algebraic structures endowed with a topology, including:
non-Archimedean Polish $R$-modules, locally compact Polish $R$-modules,
locally bounded vector spaces over a Polish field, separable Banach spaces
and separable Fr\'{e}chet spaces over a separable non-Archimedean valued
field; see Theorem \ref{Theorem:left-heart-B2}.

\subsection*{Acknowledgments}

We are thankful to Jeffrey Bergfalk and Matteo Casarosa for several useful
comments on a preliminary version of this article. We are also grateful to
Su Gao, Alexander Kechris, Jordan Ellenberg, Marco Moraschini, Andr\'{e}
Nies, Aristotelis Panagiotopoulos, Jonathan Rosenberg, Filippo Sarti, S\l %
awomir Solecki, and Alessio Savini for many stimulating conversations and
for their helpful remarks.

\section{Category-theory background\label{Section:category}}

In this section we recall some notions and results from category theory that
are needed in the rest of the paper. For an introduction to category theory,
see \cite{awodey_category_2006,mac_lane_categories_1998}.

\subsection{Additive categories}

Recall that a \emph{preadditive category}, also called an $\mathrm{Ab}$%
-category, is a category $\mathcal{C}$ in which each hom-set \textrm{Hom}$_{%
\mathcal{C}}\left( A,B\right) $ for objects $A$ and $B$ is an abelian group,
in such a way that composition of morphisms is bilinear \cite[Section I.8]%
{mac_lane_categories_1998}. Thus, for morphisms $f_{0},f_{1}:A\rightarrow B$
and $g_{0},g_{1}:B\rightarrow C$ in $\mathcal{C}$, one has that%
\begin{equation*}
\left( g_{0}+g_{1}\right) \circ \left( f_{0}+f_{1}\right) =g_{0}\circ
f_{0}+g_{0}\circ f_{1}+g_{1}\circ f_{0}+g_{1}\circ f_{1}\text{.}
\end{equation*}%
In a preadditive category, binary products and binary coproducts coincide,
and are called \emph{biproducts}. Furthermore, an object $X$ is initial if
and only if it is terminal if and only if $1_{X}\mathrm{\ }$is the zero
element of the abelian group $\mathrm{Hom}_{\mathcal{C}}\left( X,X\right) $,
in which case $X$ is called a \emph{zero object}; see \cite[Section VIII.2]%
{mac_lane_categories_1998} and \cite[Section IX.1]{mac_lane_homology_1995}.
An \emph{additive category }is a preadditive category that has a \emph{zero
object}, denoted by $0$, and such that every pair of objects $A,B$ has a
biproduct, denoted by $A\oplus B$; see \cite[Section VIII.2]%
{mac_lane_categories_1998}. A functor $F:\mathcal{C}\rightarrow \mathcal{D}$
between additive categories is called additive if satisfies $F\left(
f_{0}+f_{1}\right) =F\left( f_{0}\right) +F\left( f_{1}\right) $ whenever $%
f_{0},f_{1}:A\rightarrow B$ are morphisms in $\mathcal{C}$. This is
equivalent to the assertion that $F$ preserves biproducts of pairs of
objects of $\mathcal{C}$; see \cite[Section VIII.2, Proposition 3]%
{mac_lane_categories_1998}.

An \emph{additive subcategory }$\mathcal{B}$ of an additive category $%
\mathcal{A}$ is a (not necessarily full) subcategory of $\mathcal{A}$ that
is also an additive category, and such that the inclusion functor $\mathcal{A%
}\rightarrow \mathcal{B}$ is additive.

\subsection{Quasi-abelian categories}

A \emph{quasi-abelian category} \cite[Definition 4.1]{buhler_exact_2010}
(called \emph{almost abelian} in \cite{rump_almost_2001}) is an additive
category such that:

\begin{enumerate}
\item every morphism has a kernel and a cokernel;

\item the class of kernels is stable under push-out along arbitrary
morphisms, and the class of cokernels is stable under pull-back along
arbitrary morphisms;
\end{enumerate}

see also \cite{schneiders_quasi-abelian_1999}. In a quasi-abelian category,
one defines the image $\mathrm{im}\left( f\right) $ of an arrow $%
f:A\rightarrow B$ to be the subobject $\mathrm{ker}\left( \mathrm{coker}%
\left( f\right) \right) $ of $B$, and the coimage $\mathrm{coim}\left(
f\right) $ to be the quotient $\mathrm{coker}\left( \mathrm{ker}\left(
f\right) \right) $ of $A$ \cite[Definition 4.6]{buhler_exact_2010}. Then $f$
induces a unique arrow $\widehat{f}:\mathrm{coim}\left( f\right) \rightarrow 
\mathrm{im}\left( f\right) $ such that 
\begin{equation*}
\mathrm{im}\left( f\right) \circ \widehat{f}\circ \mathrm{coim}\left(
f\right) =f\text{.}
\end{equation*}%
Such an arrow $\widehat{f}$ is both monic and epic \cite[Proposition 4.8]%
{buhler_exact_2010}. By definition, the arrow $f$ is \emph{strict }if $%
\widehat{f}$ is an isomorphism \cite[Definition 1.1.1]%
{schneiders_quasi-abelian_1999}. One has that:

\begin{itemize}
\item an arrow is a kernel if and only if it is monic and strict;

\item an arrow is a cokernel if and only if it is epic and strict;

\item an arrow $f$ is strict if and only if it has a factorization $f=me$
where $m$ is a strict monomorphism and $e$ is a strict epimorphism;

\item the composition of strict epic arrows is strict \cite[Proposition 1.1.7%
]{schneiders_quasi-abelian_1999};

\item the composition of strict monic arrows is strict.
\end{itemize}

Considering the expression of limits in terms of products and equalizers 
\cite[Proposition 5.21]{awodey_category_2006}, we have that a quasi-abelian
category is \emph{finitely complete}, i.e.\ it has all finite limits. Since
the opposite of a quasi-abelian category is also quasi-abelian, by duality a
quasi-abelian category also has all finite colimits.

\subsection{The left heart of a quasi-abelian category}

An \emph{abelian category }is a quasi-abelian category $\mathcal{M}$ such
that every monic arrow is a kernel, and every epic arrow is a cokernel or,
equivalently, every arrow is strict \cite[Section VIII.3]%
{mac_lane_categories_1998}; see also \cite[Section IX.2]%
{mac_lane_homology_1995}.

A sequence 
\begin{equation*}
0\rightarrow A\overset{f}{\rightarrow }B\overset{g}{\rightarrow }%
C\rightarrow 0
\end{equation*}%
in a quasi-abelian category is \emph{short-exact} or a \emph{kernel-cokernel
pair }if $f$ is a kernel of $g$ and $g$ is a cokernel of $f$. A sequence 
\begin{equation*}
0\rightarrow A\overset{f}{\rightarrow }B\overset{g}{\rightarrow }C
\end{equation*}%
is left short-exact if $f$ is a kernel of $g$, while a sequence%
\begin{equation*}
A\overset{f}{\rightarrow }B\overset{g}{\rightarrow }C\rightarrow 0
\end{equation*}%
is right short-exact if $g$ is a cokernel of $f$.

A functor $F:\mathcal{A}\rightarrow \mathcal{B}$ from a quasi-abelian
category $\mathcal{A}$ to an abelian category $\mathcal{B}$ is:

\begin{itemize}
\item \emph{left exact} if for every short-exact sequence 
\begin{equation*}
0\rightarrow A\overset{f}{\rightarrow }B\overset{g}{\rightarrow }%
C\rightarrow 0,
\end{equation*}%
the sequence%
\begin{equation*}
0\rightarrow F\left( A\right) \overset{F\left( f\right) }{\rightarrow }%
F\left( B\right) \overset{F\left( g\right) }{\rightarrow }F\left( C\right)
\end{equation*}%
is left short-exact or, equivalently, $F$ preserves the kernels of strict
arrows;

\item \emph{strongly left exact} if for every left short-exact sequence 
\begin{equation*}
0\rightarrow A\overset{f}{\rightarrow }B\overset{g}{\rightarrow }C
\end{equation*}%
the sequence%
\begin{equation*}
0\rightarrow F\left( A\right) \overset{F\left( f\right) }{\rightarrow }%
F\left( B\right) \overset{F\left( g\right) }{\rightarrow }F\left( C\right)
\end{equation*}%
is left short-exact or, equivalently, $F$ preserves the kernels of arbitrary
arrows;

\item \emph{right exact} if for every short-exact sequence 
\begin{equation*}
0\rightarrow A\overset{f}{\rightarrow }B\overset{g}{\rightarrow }%
C\rightarrow 0,
\end{equation*}%
the sequence%
\begin{equation*}
F\left( A\right) \overset{F\left( f\right) }{\rightarrow }F\left( B\right) 
\overset{F\left( g\right) }{\rightarrow }F\left( C\right) \rightarrow 0
\end{equation*}%
is right short-exact or, equivalently, $F$ preserves the cokernel of strict
arrows;

\item \emph{strongly right exact} if for every short-exact sequence 
\begin{equation*}
A\overset{f}{\rightarrow }B\overset{g}{\rightarrow }C\rightarrow 0,
\end{equation*}%
the sequence%
\begin{equation*}
F\left( A\right) \overset{F\left( f\right) }{\rightarrow }F\left( B\right) 
\overset{F\left( g\right) }{\rightarrow }F\left( C\right) \rightarrow 0
\end{equation*}%
is right short-exact or, equivalently, $F$ preserves the cokernel of
arbitrary arrows

\item \emph{exact }if it is both left and right exact;
\end{itemize}

see \cite[Section 1]{rump_almost_2001} and \cite[Section 1.5]%
{stein_universal_2014}.

\begin{lemma}
Let $\mathcal{A}$ and $\mathcal{B}$ be a quasi-abelian categories, and $F:%
\mathcal{A}\rightarrow \mathcal{B}$ be a functor. The following assertions
are equivalent:

\begin{enumerate}
\item $F$ is finitely continuous;

\item $F$ is additive, left exact, and preserves monomorphisms;

\item $F$ is additive and strongly left exact.
\end{enumerate}

If $\mathcal{A}$ is abelian, these are also equivalent to:

\begin{enumerate}
\item[(4)] $F$ is additive and left exact.
\end{enumerate}
\end{lemma}

\begin{proof}
(1)$\Rightarrow $(2) If $F$ preserves all finite limits, then in particular
it preserves kernels and biproducts. Thus, it preserves monomorphisms, since
an arrow $f$ is monic if and only if the kernel of $f$ is zero. Furthermore,
considering that $F$ preserves the kernel of strict epimorphisms, we have
that $F$ is left exact.

(2)$\Rightarrow $(3) Suppose that $F$ is additive, left exact, and preserves
monomorphisms. We claim that $F$ preserves kernels of arbitrary morphisms.
Suppose that $f:A\rightarrow B$ is a morphism in $\mathcal{A}$ and let $%
e:E\rightarrow A$ be a kernel of $f$. Consider the canonical decomposition $%
f=k\circ j$ as in \cite[Proposition 1.1.14]{schneiders_quasi-abelian_1999},
where $j:A\rightarrow \mathrm{coim}\left( f\right) $ is a cokernel and $k:%
\mathrm{coim}\left( f\right) \rightarrow B$ is a \emph{monomorphism }(which
is not necessarily a kernel). Since $F$ preserves monomorphisms, we have
that $F\left( k\right) $ is a monomorphism.

Since $k$ is monic, we have that $e$ is a kernel of $j$. Since $j$ is a
cokernel and $F$ is left exact, we have that $F\left( e\right) $ is a kernel
of $F\left( j\right) $. Since $F\left( k\right) $ is monic, $F\left(
e\right) $ is a kernel of $F\left( k\right) \circ F\left( j\right) =F\left(
f\right) $. This concludes the proof that $F$ preserves kernels of arbitrary
morphisms.

(3)$\Rightarrow $(1) Since $F$ is additive, it preserves biproducts and the
zero object. Hence, by induction, it preserves finite products. Since $F$ is
strongly left exact, it preserves kernels. Hence, being additive, it also
preserves equalizers. Considering the expression of finite limits in terms
of finite products and equalizers, we have that $F$ preserves all finite
limits \cite[Proposition 5.21]{awodey_category_2006}.

Finally, if $\mathcal{A}$ is abelian, then every arrow in $\mathcal{A}$ is
strict, and a left exact functor is also strongly exact.
\end{proof}

\begin{corollary}
Let $\mathcal{A}$ and $\mathcal{B}$ be quasi-abelian categories, and $F:%
\mathcal{A}\rightarrow \mathcal{B}$ be a functor. The following assertions
are equivalent:

\begin{enumerate}
\item $F$ is exact and finitely continuous;

\item $F$ is additive, exact, and preserves monomorphisms;

\item $F$ is additive, right exact, and strongly left exact.
\end{enumerate}

If $\mathcal{A}$ is abelian, these are also equivalent to:

\begin{enumerate}
\item[(4)] $F$ is additive and exact.
\end{enumerate}
\end{corollary}

If $\mathcal{A}$ is a abelian category, then an \emph{abelian subcategory }%
of $\mathcal{A}$ is a (not necessarily full) subcategory $\mathcal{B}$ that
is also an abelian category and such that the inclusion functor is additive
and exact. Similarly, if $\mathcal{A}$ is a quasi-abelian category, then a 
\emph{quasi-abelian subcategory }of $\mathcal{A}$ is a (not necessarily
full) subcategory $\mathcal{B}$ that is also a quasi-abelian category, and
such that the inclusion functor $\mathcal{B}\rightarrow \mathcal{A}$ is
finitely continuous and finitely cocontinuous.

Let $\mathcal{A}$ be a quasi-abelian category.\ Then there exists an
essentially unique (left) \textquotedblleft completion\textquotedblright\ of 
$\mathcal{A}$ to an abelian category. This is constructed:

\begin{itemize}
\item in \cite[Section 1.2.4]{schneiders_quasi-abelian_1999}, building on 
\cite{beilinson_faisceaux_1982}, and more generally for \emph{additive
regular categories }in \cite{henrard_left_2021}---see also \cite[Chapter III]%
{buhler_algebraic_2011}---where it is called the \emph{left heart }(\emph{%
coeur}) of (the derived category of) $\mathcal{A}$ and denoted by $\mathcal{%
LH}(\mathcal{A})$;

\item in \cite[Section 3]{rump_almost_2001}, under the weaker assumption
that $\mathcal{A}$ is \emph{left quasi-abelian}, where it is called the 
\emph{left abelian cover }of $\mathcal{A}$ and denoted by $\mathrm{Q}_{l}(%
\mathcal{A})$;

\item in \cite{bodzenta_abelian_2020}, in the more general context of \emph{%
exact categories}, where it is called the \emph{right abelian envelope }of $%
\mathcal{A}$ and denoted by $\mathcal{A}_{r}(\mathcal{A})$;

\item in \cite[Section 3]{rump_abelian_2020}, in the more general context of 
\emph{left exact categories}, where it is denoted by $\mathbf{Q}_{\ell }(%
\mathcal{A})$ and called the \emph{left quotient category }of $\mathcal{A}$,
as in this case $\mathbf{Q}_{\ell }(\mathcal{A})$ is left abelian but not
necessarily abelian.
\end{itemize}

Following \cite{schneiders_quasi-abelian_1999}, we will call such a category
the \emph{left heart }of $\mathcal{A}$, and denote it by $\mathrm{LH}(%
\mathcal{A})$. We collect in the following proposition the main properties
of the this category and the inclusion functor $\mathcal{A}\rightarrow 
\mathrm{LH}(\mathcal{A})$.

\begin{proposition}
\label{Proposition:left-heart}Let $\mathcal{A}$ be a quasi-abelian category,
and let $\mathcal{A}\subseteq \mathrm{LH}(\mathcal{A})$ be its left heart as
in \cite[Definition 1.2.18]{schneiders_quasi-abelian_1999}. Denote by $I:%
\mathcal{A}\rightarrow \mathrm{LH}(\mathcal{A})$ the inclusion functor, and
by $U:\mathrm{LH}(\mathcal{A})\rightarrow \mathcal{A}$ the functor given in 
\cite[Definition 1.2.24 and Definition 1.2.26]{schneiders_quasi-abelian_1999}%
. Then we have that:

\begin{itemize}
\item $I$ is finitely continuous and exact \cite[Corollary 1.2.28,
Proposition 1.2.29]{schneiders_quasi-abelian_1999};

\item The essential image of $I$ is closed under extensions, i.e.\ if 
\begin{equation*}
0\rightarrow A\rightarrow B\rightarrow C\rightarrow 0
\end{equation*}%
is a short exact sequence in $\mathrm{LH}(\mathcal{A})$ such that $A$ and $C$
are in $\mathcal{A}$, then $B$ is isomorphic to an object of $\mathcal{A}$ 
\cite[Definition 1.2.18]{schneiders_quasi-abelian_1999};

\item The essential image of $I$ is stable by subobject, i.e.\ if $%
A\rightarrow B$ is a monic arrow in $\mathrm{LH}(\mathcal{A})$ such that $B$
is in $\mathcal{A}$, then $A$ is isomorphic to an object of $\mathcal{A}$;

\item Every object $M$ of $\mathrm{LH}(\mathcal{A})$ has a \emph{presentation%
} given by a short exact sequence $0\rightarrow A_{0}\rightarrow \hat{A}%
\rightarrow M\rightarrow 0$ in $\mathrm{LH}(\mathcal{A})$, where the arrow $%
A_{0}\rightarrow \hat{A}$ is in $\mathcal{A}$, such that $M$ is isomorphic
to an object of $\mathcal{A}$ if and only if $A_{0}\rightarrow \hat{A}$ is a
strict monomorphism in $\mathcal{A}$ \cite[Section 3]{rump_almost_2001};

\item There exist a canonical isomorphism $i:U\circ I\rightarrow \mathrm{id}%
_{\mathcal{A}}$ and a canonical epimorphism $e:\mathrm{id}_{\mathrm{LH}(%
\mathcal{A})}\rightarrow I\circ U$ that establish an adjunction $U\dashv I$
witnessing that $\mathcal{A}$ is reflective subcategory of $\mathrm{LH}(%
\mathcal{A})$ \cite[Proposition 1.2.27]{schneiders_quasi-abelian_1999};

\item For every abelian category $\mathcal{M}$, the functor $I$ induces an
equivalence of categories from the category of right exact functors $\mathrm{%
LH}(\mathcal{A})\rightarrow \mathcal{M}$ to the category of right exact
functors $\mathcal{A}\rightarrow \mathcal{M}$, which restricts to an
equivalence of categories from the category of finitely continuous exact
functors $\mathrm{LH}(\mathcal{A})\rightarrow \mathcal{M}$ to the category
of finitely continuous exact functors $\mathcal{A}\rightarrow \mathcal{M}$ 
\cite[Proposition 1.2.34]{schneiders_quasi-abelian_1999}. Thus, $\mathrm{LH}%
\left( -\right) $ is the left adjoint of the inclusion functor from the
category of abelian categories and right exact functors (respectively,
finitely continuous exact functors) to the category of quasi-abelian
categories and right exact functors (respectively, finitely continuous exact
functors);

\item Let $\mathcal{B}$ be an abelian category, and let $J:\mathcal{A}%
\rightarrow \mathcal{B}$ be a fully faithful functor such that the essential
image of $J$ is closed under subobjects, and such that for every object $B$
of $\mathcal{B}$ there exists an epimorphism $J\left( A\right) \rightarrow B$
in $\mathcal{B}$ for some object $A$ of $\mathcal{A}$. Then $J$ extends to
an equivalence of categories $\mathrm{LH}(\mathcal{A})\rightarrow \mathcal{B}
$ \cite[Proposition 1.2.36]{schneiders_quasi-abelian_1999}.
\end{itemize}
\end{proposition}

Recall that a category is \emph{countably complete} if it has all countable
limits. A quasi-abelian category is countably complete if and only if it has
countable products. A functor between countably complete categories is \emph{%
countably continuous} if it commutes with countable limits. One can show
that the left heart of a countably complete quasi-abelian category is
countably complete. In fact, we have the following result:

\begin{proposition}
Let $\mathcal{A}$ be a quasi-abelian category. If $\mathcal{A}$ is countably
complete, then $\mathrm{LH}\left( \mathcal{A}\right) $ is countably
complete, and the inclusion functor $I:\mathcal{A}\rightarrow \mathrm{LH}%
\left( \mathcal{A}\right) $ is countably continuous.

Suppose that $\mathcal{M}$ is a countably complete abelian category. Let $F:%
\mathcal{A}\rightarrow \mathcal{M}$ be a finitely continuous exact functor,
and let $\hat{F}:\mathrm{LH}\left( \mathcal{A}\right) \rightarrow \mathcal{M}
$ be the (unique up to isomorphism) finitely continuous exact functor such
that $\hat{F}\circ I$ is isomorphic to $F$. If $F$ is countably continuous,
then $\hat{F}$ is countably continuous.
\end{proposition}

\section{Abelian groups with a Polish cover\label{Section:polish-cover}}

\subsection{Borel-definable groups}

We present here the notion of Borel-definable set and Borel-definable group
as in \cite%
{lupini_definable_2020,lupini_definable_2020-1,bergfalk_definable_2022}. We
begin with recalling the notion of \emph{idealistic }equivalence relation
from \cite{kechris_countable_1994}; see also \cite[Definition 5.4.9]%
{gao_invariant_2009} and \cite{kechris_borel_2016}. We will consider as in 
\cite{lupini_definable_2020} a slightly more generous notion. Recall that a $%
\sigma $-filter on a set $C$ is a nonempty family $\mathcal{F}$ of nonempty
subsets of $C$ that is closed under countable intersections and such that $%
A\subseteq B\subseteq C$ and $A\in \mathcal{F}$ implies $B\in \mathcal{F}$.

\begin{definition}
Suppose that $X$ is a standard Borel space, and $E$ is an equivalence
relation on $X$. We say that $E$ is \emph{idealistic }if there exist a Borel
function $s:X\rightarrow X$ and an assignment $C\mapsto \mathcal{F}_{C}$
mapping each $E$-class $C$ to a $\sigma $-filter $\mathcal{F}_{C}$ on $C$
such that $s\left( x\right) Ex$ for every $x\in X$, and for every Borel
subset $A\subseteq X\times X$,%
\begin{equation*}
A_{s,\mathcal{F}}:=\left\{ x\in X:\left\{ x^{\prime }\in \left[ x\right]
_{E}:\left( s\left( x\right) ,x^{\prime }\right) \in A\right\} \in \mathcal{F%
}_{[x]_{E}}\right\}
\end{equation*}%
is a Borel subset of $X$.
\end{definition}

The term idealistic is due to the fact that it can be equivalently defined
in terms of $\sigma $-ideals, in view of the duality between $\sigma $%
-ideals and $\sigma $-filters. If $X$ is a Polish space endowed with a
continuous action of a Polish group $G$, then the corresponding orbit
equivalence relation $E_{G}^{X}$ is idealistic \cite[Proposition 5.4.10]%
{gao_invariant_2009}.

A \emph{Borel-definable set} is a pair $(\hat{X},E)$ where $\hat{X}$ is a
Polish space and $E$ is a Borel and idealistic equivalence relation on $\hat{%
X}$. We denote such a Borel-definable set by $X=\hat{X}/E$, as we think of
it as an explicit presentation of the set $X$ as a quotient of the Polish
space $\hat{X}$ by the \textquotedblleft well-behaved\textquotedblright\
equivalence relation $E$. We identify a Polish space $\hat{X}$ with the
Borel-definable set $X=\hat{X}/E$ where $E$ is the identity relation on $%
\hat{X}$. A subset $Z$ of a Borel-definable set $X=\hat{X}/E$ is Borel if $%
\hat{Z}:=\{x\in \hat{X}:[x]_{E}\in Z\}$ is a Borel subset of $\hat{X}$, in
which case $Z$ is itself a Borel-definable set. Similarly, we say that $Z$
is $\boldsymbol{\Sigma }_{1}^{1}$ or analytic if $\hat{Z}$ is an analytic
subset of $\hat{X}$. If $X=\hat{X}/E$ and $Y=\hat{Y}/F$ are Borel-definable
sets, then we define $X\times Y$ to be the Borel-definable set $(\hat{X}%
\times \hat{Y})/\left( E\times F\right) $, where $\left( x,y\right) \left(
E\times F\right) \left( x^{\prime },y^{\prime }\right) \Leftrightarrow
xEx^{\prime }$ and $yFy^{\prime }$. (One can verify that $E\times F$ is
Borel and idealistic whenever both $E$ and $F$ are Borel and idealistic.)

Suppose that $X=\hat{X}/E$ and $Y=\hat{Y}/F$ are Borel-definable sets, and $%
f:X\rightarrow Y$ is a function. A \emph{lift} $\hat{f}$ of $f$ is a
function $\hat{f}:\hat{X}\rightarrow \hat{Y}$ such that $f\left( \left[ x%
\right] _{E}\right) =[\hat{f}\left( x\right) ]_{F}$ for every $x\in X$. In
this case, we also say that $f$ is induced by $\hat{f}$.

\begin{definition}
\label{Definition:Borel-definable}Suppose that $X=\hat{X}/E$ and $Y=\hat{Y}%
/F $ are Borel-definable sets, and $f:X\rightarrow Y$ is a function. Then $f$
is \emph{Borel-definable }if it admits a Borel lift $\hat{f}:\hat{X}%
\rightarrow \hat{Y}$.
\end{definition}

By the version of the selection theorem \cite[Theorem 18.6]%
{kechris_classical_1995} presented in the proof of \cite[Lemma 3.7]%
{kechris_borel_2016}, the function $f:X\rightarrow Y$ is Borel-definable if
and only if its graph is a Borel subset of $X\times Y$.

The category of Borel-definable sets has Borel-definable functions as
morphisms. This category contains the category of standard Borel spaces and
Borel functions as a full subcategory, and it satisfies natural
generalizations of several good properties of the latter. We recall here the
most salient ones; see \cite[Proposition 1.10]{lupini_definable_2020} and
references therein.

\begin{proposition}[Kechris--Macdonald \protect\cite{kechris_borel_2016}]
\label{Proposition:Kechris--Macdonald}Suppose that $X$ and $Y$ are
Borel-definable sets. If $f:X\rightarrow Y$ is a Borel-definable injection,
and $A\subseteq X$ is Borel, then $f\left( A\right) \subseteq Y$ is Borel,
and the inverse function $f^{-1}:f\left( A\right) \rightarrow Y$ is
Borel-definable.
\end{proposition}

\begin{proposition}[Motto Ros \protect\cite{motto_ros_complexity_2012}]
\label{Proposition:Motto Ros}Suppose that $X$ and $Y$ are Borel-definable
sets. If there exist a Borel-definable injection $X\rightarrow Y$ and a
Borel-definable injection $Y\rightarrow X$, then there exists a
Borel-definable bijection $X\leftrightarrow Y$.
\end{proposition}

If $X=\hat{X}/E$ and $Y=\hat{Y}/F$ are Borel-definable sets, then the
Borel-definable set $X\times Y$ as defined above is the product of $X$ and $%
Y $ in the category of Borel-definable sets.

More generally, one can consider sets that are presented as $X=\hat{X}/E$
where $\hat{X}$ is a Polish space and $E$ is a analytic equivalence relation
on $\hat{X}$ that is not necessarily Borel or idealistic. In this case, we
say that $X$ is a $\boldsymbol{\Sigma }_{1}^{1}$\emph{-definable set}. The
notions of Borel and analytic subset of a $\boldsymbol{\Sigma }_{1}^{1}$%
-definable set, and of Borel-definable function between $\boldsymbol{\Sigma }%
_{1}^{1}$-definable sets, can be formulated as in the case of
Borel-definable sets. The category of $\boldsymbol{\Sigma }_{1}^{1}$%
-definable sets has $\boldsymbol{\Sigma }_{1}^{1}$-definable functions
(which are the functions with analytic graph) as morphisms. The following
result is established in \cite[Corollary 1.14]{lupini_definable_2020-1}.

\begin{lemma}
\label{Lemma:Borel-definable-set}Suppose that $X$ is a Borel-definable set
and $Y$ is a $\boldsymbol{\Sigma }_{1}^{1}$-definable set. If there exists a
Borel-definable bijection $f:X\rightarrow Y$, then $Y$ is Borel-definable.
\end{lemma}

A \emph{Borel-definable group} is simply a group object in the category of
Borel-definable sets in the sense of \cite[Section III.6]%
{mac_lane_categories_1998}. Explicitly, a Borel-definable group is a
Borel-definable set $G$ that is also a group, and such that the group
operation $G\times G\rightarrow G$ and the function $G\rightarrow G$ mapping
each element to its inverse are Borel-definable. Notice that every standard
Borel group is, in particular, a Borel-definable group. Naturally, a $%
\boldsymbol{\Sigma }_{1}^{1}$-definable group is a group object in the
category of $\boldsymbol{\Sigma }_{1}^{1}$-definable sets.

An important example of $\boldsymbol{\Sigma }_{1}^{1}$-definable group that
is not a Borel-definable group is $\mathbb{R}^{\omega }/E_{1}$, where $%
\mathbb{R}^{\omega }$ is the product of countably many copies of the Polish
group $\mathbb{R}$, and $E_{1}$ is the tail-equivalence relation on $\mathbb{%
R}^{\omega }$, obtained by setting $\left( \boldsymbol{x}_{i}\right)
E_{1}\left( \boldsymbol{y}_{i}\right) \Leftrightarrow \exists n\forall i\geq
n,$ $x_{i}=y_{i}$. Notice that, if $\mathbb{R}^{\left( \omega \right)
}\subseteq \mathbb{R}^{\omega }$ is the subgroup consisting of sequences
that are eventually zero, then $E_{1}$ is the coset equivalence relation
associated with $\mathbb{R}^{\left( \omega \right) }$. Thus, $\mathbb{R}%
^{\omega }/E_{1}$ can be seen as the quotient group $\mathbb{R}^{\omega }/%
\mathbb{R}^{\left( \omega \right) }$. It is proved in \cite[Theorem 4.1]%
{kechris_classification_1997} that if $X$ is a Borel-definable set, then
there is no Borel-definable injection $\mathbb{R}^{\omega }/\mathbb{R}%
^{\left( \omega \right) }\rightarrow X$.

\subsection{Groups with a Polish cover}

We now recall the notion of abelian group with a Polish cover introduced in 
\cite{bergfalk_definable_2020}.

\begin{definition}
\label{Definition:Polish cover}An abelian group with a Polish cover is a
Borel-definable abelian group given as a quotient $\hat{G}/N$ where $\hat{G}$
is an abelian Polish group and $N\subseteq \hat{G}$ is a Polishable
subgroup. This means that $N$ is a Borel subgroup of $\hat{G}$ such that
there is a Polish group topology on $N$ whose open sets are Borel in $\hat{G}
$ or, equivalently, there exist a Polish group $H$ and a continuous
homomorphism $\psi :H\rightarrow \hat{G}$ with image $N$. For $x,y\in \hat{G}
$, we write $x\equiv y\mathrm{\ \mathrm{mod}}\ N$ if $x-y\in N$.
\end{definition}

We regard an abelian Polish group $G$ as an abelian group with a Polish
cover $\hat{G}/N$ where $G=\hat{G}$ and $N=\left\{ 0\right\} $. If $G=\hat{G}%
/N$ and $H=\hat{H}/M$ are abelian groups with a Polish cover, then we define 
$G\oplus H$ to be the abelian group with a Polish cover $(\hat{G}\oplus \hat{%
H})/(N\oplus M)$. Similarly, if $G_{k}=\hat{G}_{k}/N_{k}$ is an abelian
group with a Polish cover for $k\in \omega $, then we define $\prod_{k\in
\omega }G_{k}$ to be the abelian group with a Polish cover $\hat{G}/N$ where 
$\hat{G}=\prod_{k}\hat{G}_{k}$ and $N=\prod_{k}N_{k}$.

Recall that a Polish group is \emph{non-Archimedean} if it has a basis of
neighborhoods of the identity consisting of \emph{open subgroups}; see \cite[%
Proposition 2.1]{ding_non-archimedean_2017} for other characterizations. A
Polish group is \emph{locally compact }if it has a basis of neighborhoods of
the identity consisting of \emph{compact} subsets.

\begin{definition}
\label{Definition:non-A}Suppose that $G=\hat{G}/N$ is an abelian group with
a Polish cover. Then we say that $G$ is:

\begin{itemize}
\item an abelian group with a non-Archimedean Polish cover if $\hat{G}$ and $%
N$ are non-Archimedean Polish groups;

\item an abelian group with a locally compact Polish cover if $\hat{G}$ and $%
N$ are locally compact Polish groups.
\end{itemize}
\end{definition}

As an abelian group with a Polish cover is, in particular, a Borel-definable
group, the notion of Borel-definable homomorphism between abelian groups
with a Polish cover is a particular instance of the notion of
Borel-definable group homomorphism between Borel-definable groups.

\begin{definition}
\label{Definition:Borel-definable homomorphism}Suppose that $G=\hat{G}/N$
and $H=\hat{H}/M$ are abelian groups with a Polish cover. A group
homomorphism $f:G\rightarrow H$ is:

\begin{itemize}
\item \emph{Borel-definable} if it has a Borel lift $\hat{f}:\hat{G}%
\rightarrow \hat{H}$;

\item \emph{Baire-definable }if it has a Baire-measurable lift $\hat{f}:\hat{%
G}\rightarrow \hat{H}$ \cite[Definition 8.37]{kechris_classical_1995};

\item \emph{continuously definable }if it has a continuous lift $\hat{f}:%
\hat{G}\rightarrow \hat{H}$;

\item \emph{locally continuously definable }if it has a Borel lift $\hat{f}:%
\hat{G}\rightarrow \hat{H}$ that is \emph{locally continuous}, namely
continuous on a zero neighborhood in $\hat{G}$;

\item \emph{liftable }if it has a lift $\hat{f}:\hat{G}\rightarrow \hat{H}$
that is a continuous group homomorphism.
\end{itemize}

If $G$ is an abelian group with a locally compact Polish cover, then we say
that $f:G\rightarrow H$ is:

\begin{itemize}
\item \emph{Haar-definable} if it has a Haar-measurable lift $\hat{f}:\hat{G}%
\rightarrow \hat{H}$.
\end{itemize}
\end{definition}

One can analogously define the notions from Definition \ref%
{Definition:Borel-definable homomorphism} in the more general context of $%
\boldsymbol{\Sigma }_{1}^{1}$-definable groups.

\begin{definition}
\label{Definition:cocycle}Suppose that $G=\hat{G}/N$ and $H=\hat{H}/M$ are
abelian groups with a Polish cover. Let $f:\hat{G}\rightarrow \hat{H}$ be
lift of a group homomorphism $G\rightarrow H$. We let $\delta f:\hat{G}%
\times \hat{G}\rightarrow M$ be the corresponding $2$-cocycle, defined by $%
\delta f\left( x,y\right) =f\left( y\right) -f\left( x+y\right) +f\left(
x\right) $.
\end{definition}

\begin{remark}
\label{Remark:inverse}It follows from Proposition \ref%
{Proposition:Kechris--Macdonald} that if a Borel-definable homomorphism $%
G\rightarrow H$ is bijective, then its inverse $H\rightarrow G$ is also a
Borel-definable homomorphism.
\end{remark}

In what follows, we consider groups with a Polish cover as objects of a
category that has Borel-definable homomorphisms as morphisms.

\begin{remark}
\label{Remark:automatic-continuity}It follows from \cite[Proposition 4.6]%
{bergfalk_definable_2020} that, \emph{when }$\hat{G}$\emph{\ is
non-Archimedean}, a group homomorphism $\varphi :\hat{G}/N\rightarrow \hat{H}%
/M$ between abelian groups with a Polish cover is Borel-definable if and
only if it is continuously definable.
\end{remark}

\begin{lemma}
\label{Lemma:locally-continuously}Let $f:\hat{G}/N\rightarrow \hat{H}/M$ be
a group homomorphism between abelian groups with a Polish cover. Let $V$ be
a zero neighborhood in $\hat{G}$ and let $\hat{f}:V\rightarrow \hat{H}$ be a
continuous function such that $\hat{f}\left( x\right) +M=f\left( x+N\right) $
for every $x\in V$. Then there exists a Borel lift for $f$ whose restriction
to $V$ is equal to $\hat{f}$.
\end{lemma}

\begin{proof}
Let $\left\{ a_{n}:n\in \omega \right\} $ be a countable dense subset of $%
\hat{G}$ with $a_{0}=0$. For $n\in \omega $, let $b_{n}\in \hat{H}$ be such
that $f\left( a_{n}+N\right) =b_{n}+M$, where $b_{0}=0$. By \cite[Theorem
18.10]{kechris_classical_1995} there exists a Borel function $\hat{G}%
\rightarrow \omega $, $x\mapsto n\left( x\right) $ such that $x\in
V+a_{n\left( x\right) }$ for every $x\in \hat{G}$ and $n\left( x\right) =0$
for $x\in V$. We can thus extend $\hat{f}$ to a Borel function on $\hat{G}$
by setting $\hat{f}\left( x\right) :=\hat{f}\left( x-a_{n\left( x\right)
}\right) +b_{n\left( x\right) }$ for $x\in \hat{X}$.
\end{proof}

\section{Subgroups of groups with a Polish cover\label{Section:subgroups}}

\subsection{Subgroups with a Polish cover\label{Subsection:subgroups}}

We now introduce in the context of abelian groups with a Polish cover the
notion of Borel subgroup and subgroup with a Polish cover.

\begin{definition}
\label{Definition:Polishable}Suppose that $G=\hat{G}/N$ is an abelian group
with a Polish cover, and $H\subseteq G$ is a subgroup. Define $\hat{H}%
=\{x\in \hat{G}:x+N\in H\}\subseteq \hat{G}$. Then we say that:

\begin{itemize}
\item $H$ is a\emph{\ Borel }(respectively, \emph{analytic}) \emph{subgroup}
of $G$ if $\hat{H}$ is a Borel (respectively, analytic) subgroup of $\hat{G}$%
;

\item $H$ is a \emph{subgroup with a Polish cover} of $G$ if $\hat{H}$ is a
Polishable subgroup of $\hat{G}$;

\item $H$ is a \emph{subgroup with a non-Archimedean Polish cover} of $G$ if 
$\hat{H}$ is a non-Archimedean Polishable subgroup of $\hat{G}$.
\end{itemize}
\end{definition}

If $H=\hat{H}/N$ is a subgroup with a Polish cover of an abelian group with
a Polish cover $G=\hat{G}/N$, where $\hat{H}$ is a Polishable subgroup of $%
\hat{G}$, then we regard $H$ as the abelian group with a Polish cover $\hat{H%
}/N$, and $G/H$ as the abelian group with a Polish cover $\hat{G}/\hat{H}$.

If $G=\hat{G}/N$ is group with a Polish cover and $H$ is a subgroup with a
Polish cover of $G$, then we let $\overline{H}^{G}$ be the closed subgroup
of $G$ obtained as the closure of $H$ in $G$ with respect to the quotient
topology induced by $\hat{G}$. We say that $H$ is dense in $G$ if $\overline{%
H}^{G}=G$.

\begin{lemma}
\label{Lemma:intersection--Polishable}Suppose that $G$ is an abelian group
with a Polish cover. Let $\left( G_{n}\right) _{n\in \omega }$ be a sequence
of subgroups with a Polish cover of $G$. Then $G_{0}+G_{1}$ and $%
\bigcap_{n\in \omega }G_{n}$ are subgroups with a Polish cover of $G$.
\end{lemma}

\begin{proof}
Write $G=\hat{G}/N$. For every $n\in \omega $, we have that $G_{n}=\hat{G}%
_{n}/N$ for some Polishable subgroup $\hat{G}_{n}$ of $\hat{G}$. We have
that 
\begin{equation*}
\{x\in \hat{G}:x+N\in G_{0}+G_{1}\}=\hat{G}_{0}+\hat{G}_{1}+N
\end{equation*}%
is the image of the Polish group $\hat{G}_{0}\oplus \hat{G}_{1}\oplus N$
under the continuous homomorphism $\hat{G}_{0}\oplus \hat{G}_{1}\oplus
N\rightarrow \hat{G}$, $\left( x,y,z\right) \mapsto x+y+z$.

Similarly, we have that 
\begin{equation*}
\{x\in \hat{G}:x+N\in \bigcap_{n\in \omega }G_{n}\}=\bigcap_{n\in \omega }%
\hat{G}_{n}
\end{equation*}%
is the image of the Polish group%
\begin{equation*}
Z:=\left\{ \left( x_{n}\right) _{n\in \omega }\in \prod_{n\in \omega }\hat{G}%
_{n}:\forall n\in \omega ,x_{n}=x_{n+1}\right\} \subseteq \prod_{n\in \omega
}\hat{G}_{n}
\end{equation*}%
under the continuous homomorphism $Z\rightarrow \hat{G}$, $\left(
x_{n}\right) _{n\in \omega }\mapsto x_{0}$.
\end{proof}

Suppose that $L$ is a Borel subgroup of an abelian group with a Polish cover 
$G$. Then one can consider the quotient $G/L$ as a $\boldsymbol{\Sigma }%
_{1}^{1}$-definable group. The implication (1)$\Rightarrow $(3) in the
following proposition can be seen as a reformulation of \cite[Theorem 1.1]%
{solecki_coset_2009}.

\begin{proposition}
\label{Proposition:E1}Suppose that $L$ is a Borel subgroup of an abelian
group with a Polish cover $G$. Consider the corresponding $\boldsymbol{%
\Sigma }_{1}^{1}$-definable group $G/L$. The following assertions are
equivalent:

\begin{enumerate}
\item there does not exist a Borel-definable injection $\mathbb{R}^{\omega }/%
\mathbb{R}^{\left( \omega \right) }\rightarrow G/L$;

\item $G/L$ is a Borel-definable group;

\item $L$ is a subgroup with a Polish cover of $G$.
\end{enumerate}
\end{proposition}

\begin{proof}
Write $G=\hat{G}/N$ and let $\hat{L}=\{x\in \hat{G}:x+N\in L\}$. Since $G/L=%
\hat{G}/\hat{L}$, after replacing $L$ with $\hat{L}$ and $G$ with $\hat{G}$,
we can assume that $G$ is in fact a Polish group.

The implication (3)$\Rightarrow $(2) follows from the fact that a group with
a Polish cover is, in particular, a Borel-definable group in view of \cite[%
Proposition 5.4.10]{gao_invariant_2009}. The implication (2)$\Rightarrow $%
(1) follows from \cite[Theorem 4.1]{kechris_classification_1997}. Finally,
the implication (1)$\Rightarrow $(3) is the content of \cite[Theorem 1.1]%
{solecki_coset_2009}.
\end{proof}

We now show that images and preimages of subgroups with a Polish cover under
Borel-definable homomorphisms are subgroups with a Polish cover.

\begin{proposition}
\label{Proposition:preimage}Suppose that $G,H$ are abelian groups with a
Polish cover, and $f:G\rightarrow H$ is a Borel-definable group homomorphism.

\begin{enumerate}
\item If $H_{0}$ is a subgroup with a Polish cover of $H$, then $%
f^{-1}\left( H_{0}\right) $ is a subgroup with a Polish cover of $G$;

\item If $G_{0}$ is a subgroup with a Polish cover of $G$, then $f\left(
G_{0}\right) $ is a subgroup with a Polish cover of $G$.
\end{enumerate}
\end{proposition}

\begin{proof}
(1) After replacing $H$ with $H/H_{0}$, we can assume that $H_{0}=\left\{
0\right\} $, in which case 
\begin{equation*}
f^{-1}\left( H_{0}\right) =\mathrm{ker}\left( f\right) :=\left\{ g\in
G:f\left( g\right) =0\right\} \text{.}
\end{equation*}%
We have that $f$ induces a Borel-definable injective group homomorphism $G/%
\mathrm{ker}\left( f\right) \rightarrow H$. Notice that $\mathrm{ker}\left(
f\right) $ is a Borel subgroup of $G$. Since $H$ is a group with a Polish
cover, we have that there does not exist a Borel-definable injection $%
\mathbb{R}^{\omega }/\mathbb{R}^{\left( \omega \right) }\rightarrow H$ by
Proposition \ref{Proposition:E1}. Thus, there does not exist a
Borel-definable injection $\mathbb{R}^{\omega }/\mathbb{R}^{\left( \omega
\right) }\rightarrow G/\mathrm{ker}\left( f\right) $. Thus, $\mathrm{ker}%
\left( f\right) $ is a subgroup with a Polish cover of $G$ by Proposition %
\ref{Proposition:E1} again.

(2) After replacing $G$ with $G_{0}$ and $f$ with its restriction to $G_{0}$%
, we can assume that $G=G_{0}$. By the first item, $\mathrm{ker}\left(
f\right) $ is a subgroup with a Polish cover of $G$. Thus, after replacing $%
G $ with $G/\mathrm{ker}\left( f\right) $, we can assume that $f$ is
injective. In this case, we have that $f\left( G\right) $ is a Borel
subgroup of the Borel-definable group $H$ by Proposition \ref%
{Proposition:Kechris--Macdonald}. By Proposition \ref{Proposition:E1}, to
conclude the proof it suffices to prove that $H/f\left( G\right) $ is a
Borel-definable group.

Write $G=\hat{G}/N$ and $H=\hat{H}/M$, where $\hat{G},\hat{H}$ are Polish
groups and $N\subseteq \hat{G}$ and $M\subseteq \hat{H}$ are Polishable
subgroups. Suppose that $\varphi :\hat{G}\rightarrow \hat{H}$ is a Borel
lift of $f$. Define the Borel function $\delta \varphi :\hat{G}\times \hat{G}%
\rightarrow M$ as in Definition \ref{Definition:cocycle}. We need to prove
that the equivalence relation $E$ on $\hat{H}$ defined by setting $%
xEy\Leftrightarrow \exists \left( g,h\right) \in \hat{G}\oplus M$, $\varphi
\left( g\right) +h+x=y$ is idealistic. The argument is similar to the one
from \cite[Proposition 5.4.10]{gao_invariant_2009}. We adopt the notation of
category quantifiers as in \cite[Section 16]{kechris_classical_1995}. For $%
x\in \hat{H}$ and $A\subseteq \left[ x\right] _{E}$, we set $A\in \mathcal{F}%
_{\left[ x\right] _{E}}\Leftrightarrow \forall ^{\ast }g\in \hat{G}$, $%
\forall ^{\ast }h\in M$, $\varphi \left( g\right) +h+x\in A$. Observe that $%
\mathcal{F}_{\left[ x\right] }$ does not depend on the choice of the
representative $x$ for the equivalence class $\left[ x\right] _{E}$. Indeed,
if $x_{0}Ex$ then there exists $\left( g_{0},h_{0}\right) \in \hat{G}\times
M $ such that $\varphi \left( g_{0}\right) +h_{0}+x_{0}=x$.\ If $A\in 
\mathcal{F}_{\left[ x\right] }$ then $\forall ^{\ast }g\in \hat{G}$, $%
\forall ^{\ast }h\in M$, $\varphi \left( g\right) +h+x\in A$. We have that%
\begin{eqnarray*}
\varphi \left( g\right) +h+x &=&\varphi \left( g\right) +h+\varphi \left(
g_{0}\right) +h_{0}+x_{0} \\
&=&\varphi \left( g+g_{0}\right) +h+h_{0}+\delta \varphi \left(
g,g_{0}\right) +x_{0}\text{.}
\end{eqnarray*}%
For a fixed $\tilde{g}\in \hat{G}$, if $\forall ^{\ast }h\in M$, $\varphi
\left( \tilde{g}\right) +h+x\in A$ then $\forall ^{\ast }h\in M$, $\varphi
\left( \tilde{g}+g_{0}\right) +h+h_{0}+\delta \varphi \left( \tilde{g}%
,g_{0}\right) +x_{0}\in A$. Since $h_{0}+\delta \varphi \left( \tilde{g}%
,g_{0}\right) \in M$, this implies that $\forall ^{\ast }h\in M$, $\varphi
\left( \tilde{g}+g_{0}\right) +h+x_{0}\in A$. Therefore, we have that $%
\forall ^{\ast }g\in \hat{G}$, $\forall ^{\ast }h\in M$, $\varphi \left(
g+g_{0}\right) +h+x_{0}\in A$ and hence $\forall ^{\ast }g\in \hat{G}$, $%
\forall ^{\ast }h\in M$, $\varphi \left( g\right) +h+x_{0}\in A$. This shows
that $\mathcal{F}_{\left[ x\right] _{E}}$ is well-defined. It is easy to
verify that $\mathcal{F}_{\left[ x\right] _{E}}$ is a $\sigma $-filter on $%
\left[ x\right] _{E}$. It remains to prove that if $A\subseteq \hat{H}\times 
\hat{H}$ is Borel, then 
\begin{equation*}
A_{\mathcal{F}}:=\left\{ x\in \hat{H}:\left\{ y\in \hat{H}:\left( x,y\right)
\in A\right\} \in \mathcal{F}_{\left[ x\right] _{E}}\right\}
\end{equation*}%
is a Borel subset of $\hat{H}$. The argument is the same as in the proof of 
\cite[Theorem 3.3.3]{gao_invariant_2009}. We have that $x\in A_{\mathcal{F}%
}\Leftrightarrow \forall ^{\ast }g\in \hat{G}$, $\forall ^{\ast }h\in M$, $%
\left( x,\varphi \left( g\right) +h+x\right) \in A$. Define the Borel set%
\begin{equation*}
B:=\left\{ \left( g,h,x\right) \in \hat{G}\times M\times \hat{H}:\left(
x,\varphi \left( g\right) +h+x\right) \in A\right\} \text{.}
\end{equation*}%
Then we have that 
\begin{equation*}
x\in A_{\mathcal{F}}\Leftrightarrow \forall ^{\ast }g\in \hat{G},\forall
^{\ast }h\in M,\left( g,h,x\right) \in B\Leftrightarrow \forall ^{\ast
}\left( g,h\right) \in \hat{G}\times M,\left( g,h,x\right) \in B\text{.}
\end{equation*}%
Since $B$ is Borel, this shows that $A_{\mathcal{F}}$ is Borel by \cite[%
Theorem 16.1]{kechris_classical_1995}.
\end{proof}

The following corollary can be seen as a generalization of the existence of
Borel right inverses for surjective continuous group homomorphisms between
Polish groups; see \cite[Theorem 12.17]{kechris_classical_1995}.

\begin{corollary}
Suppose that $G=\hat{G}/N$ and $H=\hat{H}/M$ are abelian groups with a
Polish cover. Let $f:G\rightarrow H$ be a surjective Borel-definable group
homomorphism.\ Then there exists a Borel function $\psi :\hat{H}\rightarrow 
\hat{G}$ such that $f\left( \psi \left( h\right) +N\right) =h+M$ for every $%
h\in \hat{H}$.
\end{corollary}

\begin{proof}
We can write $f$ as the composition%
\begin{equation*}
G\overset{p}{\rightarrow }\frac{G}{\mathrm{ker}\left( f\right) }\overset{%
\bar{f}}{\rightarrow }H
\end{equation*}%
where $p$ is the quotient map and $\bar{f}$ is the Borel-definable group
isomorphism induced by $f$. It thus suffices to prove that the conclusion
holds for $p$ and $\bar{f}$. In the case of $p$ the conclusion is obvious.
In the case of $\bar{f}$, it is a consequence of Proposition \ref%
{Proposition:Kechris--Macdonald}.
\end{proof}

Suppose that $G=\hat{G}/N$ and $H=\hat{H}/M$ are abelian groups with a
Polish cover. Let $\varphi :G\rightarrow H$ be a group homomorphism. Define
the \emph{graph} $\Gamma \left( \varphi \right) $ of $\varphi $ to be the
subgroup%
\begin{equation*}
\{\left( x,y\right) \in G\oplus H:\varphi \left( x\right) =y\}\subseteq
G\oplus H\text{.}
\end{equation*}%
By Proposition \ref{Proposition:preimage}, if $\varphi $ is Borel-definable,
then $\Gamma \left( \varphi \right) $ is a subgroup with a Polish cover of $%
G\oplus H$, being the kernel of the Borel-definable homomorphism $G\oplus
G\rightarrow H$, $\left( x,y\right) \mapsto \varphi \left( x\right) -y$.
When $G$ and $H$ are abelian groups with a non-Archimedean Polish cover, $%
\Gamma \left( \varphi \right) $ is also a abelian group with a
non-Archimedean Polish cover by Theorem \ref{Theorem:left-heart-B2}(1)
below. The function $\pi _{G}:\Gamma \left( \varphi \right) \rightarrow G$, $%
\left( x,y\right) \mapsto x$ is a bijective liftable\emph{\ }group
homomorphism, and the function $\pi _{H}:\Gamma \left( \varphi \right)
\rightarrow H$, $\left( x,y\right) \mapsto y$ is a liftable group
homomorphism; see Definition \ref{Definition:Borel-definable homomorphism}.
Furthermore, we have that $\varphi =\pi _{H}\circ \left( \pi _{G}\right)
^{-1}$.

\begin{theorem}
\label{Theorem:factorization}Suppose that $\varphi :G\rightarrow H$ is a
group homomorphism between abelian groups with a Polish cover. The following
assertions are equivalent:

\begin{enumerate}
\item $\varphi $ is Borel-definable;

\item the graph $\Gamma \left( \varphi \right) $ is a subgroup with a Polish
cover of $G\oplus H$;

\item there exist an abelian group with a Polish cover $L$, a bijective
liftable group homomorphism $\sigma :L\rightarrow G$, and a liftable group
homomorphism $\psi :L\rightarrow H$ such that $\varphi =\psi \circ \sigma
^{-1}$.
\end{enumerate}
\end{theorem}

\begin{proof}
The implication (1)$\Rightarrow $(2) follows from Proposition \ref%
{Proposition:preimage} as observed above.

(2)$\Rightarrow $(3) In order to verify that (3) holds, it suffices to take $%
L$ to be the graph $\Gamma \left( \varphi \right) $ of $\varphi $, $\sigma
:\Gamma \left( \varphi \right) \rightarrow G$ to be the Borel-definable
homomorphism defined by $\left( x,y\right) \mapsto x$, and $\psi :\Gamma
\left( \varphi \right) \rightarrow H$ to be the Borel-definable homomorphism
defined by $\left( x,y\right) \mapsto y$.

(3)$\Rightarrow $(1) This follows by observing that a liftable group
homomorphism is, in particular, Borel-definable, and that the inverse of a
bijective Borel-definable group homomorphism is Borel-definable by Remark %
\ref{Remark:inverse}.
\end{proof}

\begin{proposition}
\label{Theorem:analytic-graph}Let $G=\hat{G}/N$ be an abelian group with a
Polish cover, and let $H=\hat{H}/E_{H}$ be a $\boldsymbol{\Sigma }_{1}^{1}$%
-definable group. Suppose that $\varphi :G\rightarrow H$ is a group
homomorphism. The following assertions are equivalent:

\begin{enumerate}
\item $\varphi $ is Borel-definable;

\item the graph $\Gamma \left( \varphi \right) $ is an analytic subgroup of $%
G\oplus H$;

\item $\varphi $ is Baire-definable.
\end{enumerate}

If furthermore $G$ is an abelian group with a locally compact Polish cover,
then the above conditions are equivalent to:

\begin{enumerate}
\item[(4)] $\varphi $ is Haar-definable.
\end{enumerate}
\end{proposition}

\begin{proof}
Consider the lift%
\begin{equation*}
\hat{\Gamma}\left( \varphi \right) =\{\left( x,y\right) \in \hat{G}\oplus 
\hat{H}:\varphi \left( x+N\right) E_{H}y\}\subseteq \hat{G}\oplus \hat{H}
\end{equation*}%
of $\Gamma \left( \varphi \right) \subseteq G\oplus H$.

(1)$\Rightarrow $(2) By assumption, we have that $\varphi $ has a Borel lift 
$f:\hat{G}\rightarrow \hat{H}$. We have that $\left( x,y\right) \in \hat{%
\Gamma}\left( \varphi \right) $ if and only if there exist $z\in
E_{H}\subseteq \hat{H}\times \hat{H}$ such that $\pi _{0}\left( z\right)
=f\left( x\right) $ and $\pi _{1}\left( z\right) =y$, where $\pi _{0},\pi
_{1}$ are the canonical projections from $\hat{H}\times \hat{H}$ to $\hat{H}$%
. Since by assumption $E_{H}$ is an analytic equivalence relation on $\hat{H}
$, it follows that $\hat{\Gamma}\left( \varphi \right) $ is an analytic
subset of $\hat{G}\oplus \hat{H}$, and hence by definition $\Gamma \left(
\varphi \right) $ is an analytic subgroup of $G\oplus H$.

(2)$\Rightarrow $(3) By the Jankov--von Neumann Uniformization Theorem \cite[%
Theorem 18.1]{kechris_classical_1995} applied to $\hat{\Gamma}\left( \varphi
\right) \subseteq \hat{G}\oplus \hat{H}$, we have that there exists a $%
\sigma (\boldsymbol{\Sigma }_{1}^{1})$-measurable lift $f:\hat{G}\rightarrow 
\hat{H}$ for $\varphi $. Since analytic sets are Baire-measurable \cite[%
Theorem 21.6]{kechris_classical_1995}, we have that $f$ is Baire-measurable,
and $\varphi $ is Baire-definable.

(3)$\Rightarrow $(1) Suppose that $\varphi $ is Baire-definable. Let $f:\hat{%
G}\rightarrow \hat{H}$ be a Baire-measurable lift of $\varphi $. Then there
exists a dense $G_{\delta }$ subset $C$ of $G$ such that $f|_{C}$ is
continuous. Consider the relation%
\begin{equation*}
P=\{\left( x,y\right) \in \hat{G}\times \hat{G}:y\in \left( x-C\right) \cap
C\}\text{.}
\end{equation*}%
Then by \cite[Theorem 18.6]{kechris_classical_1995}---see also \cite[proof
of Lemma 3.7]{kechris_borel_2016}---there exists a Borel function $\sigma :%
\hat{G}\rightarrow \hat{G}$ such that $\sigma \left( x\right) =x$ for $x\in
C $ and $\left( x,\sigma \left( x\right) \right) \in P$ for $x\in \hat{G}%
\setminus C$. The hypotheses of \cite[Theorem 18.6]{kechris_classical_1995}
are satisfied by \cite[Theorem 16.1]{kechris_classical_1995}, where one sets 
$\mathcal{I}_{x}$ to be the $\sigma $-ideal of meager subsets of $\hat{G}$
for every $x\in \hat{G}$. Then we have that defining%
\begin{equation*}
g\left( x\right) :=f\left( \sigma \left( x\right) \right) +f\left( x-\sigma
\left( x\right) \right)
\end{equation*}%
for $x\in \hat{G}$ yields a Borel lift for $\varphi $. This shows that $%
\varphi $ is Borel-definable.

Suppose now that $G$ is an abelian group with a locally compact Polish cover.

(4)$\Rightarrow $(1) Let $f:\hat{G}\rightarrow \hat{H}$ be a Haar-measurable
lift of $\varphi $. Then there exists a Borel set $C\subseteq \hat{G}$ such
that $\hat{G}\setminus C$ is null and $f|_{C}$ is Borel. Define%
\begin{equation*}
P=\{\left( x,y\right) \in \hat{G}\times \hat{G}:y\in \left( x-C\right) \cap
C\}\text{.}
\end{equation*}%
By \cite[Corollary 18.7]{kechris_classical_1995} there exists a Borel
function $\sigma :\hat{G}\rightarrow \hat{G}$ such that $\left( x,\sigma
\left( x\right) \right) \in P$ for every $x\in \hat{G}$. Then we have that
defining%
\begin{equation*}
h\left( x\right) :=f\left( \sigma \left( x\right) \right) +f\left( x-\sigma
\left( x\right) \right)
\end{equation*}%
for $x\in \hat{G}$ yields a Borel lift for $\varphi $.
\end{proof}

\begin{corollary}
\label{Corollary:Borel-definable-set-isomorphism}Suppose that $G$ is an
abelian group with a Polish cover, and $H$ is an abelian $\boldsymbol{\Sigma 
}_{1}^{1}$-definable group. If there exists a group isomorphism $\varphi
:G\rightarrow H$ with analytic graph, then $H$ is a Borel-definable group,
and $\varphi $ is a Borel-definable group isomorphism.
\end{corollary}

\begin{proof}
This follows immediately from Lemma \ref{Lemma:Borel-definable-set} and
Proposition \ref{Theorem:analytic-graph}.
\end{proof}

\subsection{Complexity of subgroups\label{Section:complexity}}

In this section, we consider the complexity of subgroups of groups with a
Polish cover. We reformulate in this context some results from \cite%
{hjorth_borel_1998,farah_borel_2006,lupini_complexity_2022}. Recall that a
Borel \emph{complexity class }$\Gamma $ is an assignment $X\mapsto \Gamma
\left( X\right) $ from Polish spaces to classes of Borel sets such that for
every continuous function $f:X\rightarrow Y$ between Polish spaces $X,Y$ and
for every $A\in \Gamma \left( Y\right) $, $A\subseteq Y$ and $f^{-1}\left(
A\right) \in \Gamma \left( X\right) $. Given such a complexity class, its
dual class $\check{\Gamma}$ is defined by setting $\check{\Gamma}\left(
X\right) =\left\{ X\setminus A:A\in \Gamma \left( X\right) \right\} $ for
every Polish space $X$. A complexity class $\Gamma $ is \emph{not self-dual}
if it is different from $\check{\Gamma}$. We will be mostly concerned with
the complexity classes $\boldsymbol{\Sigma }_{\alpha }^{0}$, $\boldsymbol{%
\Sigma }_{\alpha }^{0}$, $D(\boldsymbol{\Sigma }_{\alpha }^{0})$, $%
\boldsymbol{\Delta }_{\alpha }^{0}$ for $1\leq \alpha <\omega _{1}$ and
their duals; see \cite[Section 11.B]{kechris_classical_1995}.

\begin{definition}
Suppose that $H$ is a subgroup of an abelian group with a Polish cover $G=%
\hat{G}/N$. Set $\hat{H}=\{x\in \hat{G}:x+N\in H\}$. Let $\Gamma $ be a
complexity class. We say that $H$ belongs to $\Gamma (G)$ or that $H$ is $%
\Gamma $ in $G$ if and only if $\hat{H}\in \Gamma (\hat{G})$. If $\Gamma $
is not self-dual, then we say that $\Gamma $ is the complexity class of $H$
in $G$ if and only if $\hat{H}\in \Gamma (\hat{G})$ and $\hat{H}\notin 
\check{\Gamma}(\hat{G})$.
\end{definition}

Given a Borel-definable set $X=\hat{X}/E$, we denote by $=_{X}$ the
equivalence relation $E$. In particular, if $G=\hat{G}/N$ is a group with a
Polish cover, and $H=\hat{H}/N$ is a subgroup with a Polish cover of $G$,
then $=_{G/H}$ is the coset relation of $\hat{H}$ inside of $\hat{G}$.
Recall that an equivalence relation $E$ on a Polish space $X$ is \emph{%
potentially }$\Gamma $ if it is Borel-reducible to an equivalence relation $%
F $ on a Polish space $Y$ such that $F\in \Gamma \left( Y\times Y\right) $.

As a consequence of \cite[Theorem 1.1 and Theorem 1.2]%
{lupini_complexity_2022} we have the following results.

\begin{proposition}
\label{Proposition:complexity-classes}The following is a complete list of
possible Borel complexity classes of subgroups with a Polish cover of
abelian groups with a Polish cover: $\boldsymbol{\Pi }_{1+\lambda }^{0}$, $%
\boldsymbol{\Sigma }_{1+\lambda +1}^{0}$, $D(\boldsymbol{\Pi }_{1+\lambda
+n+1}^{0})$, and $\boldsymbol{\Pi }_{1+\lambda +n+2}^{0}$ for $\lambda
<\omega _{1}$ either zero or a limit ordinal, and $n<\omega $.
\end{proposition}

\begin{proposition}
\label{Proposition:non-Archimedean-complexity-classes}The following is a
complete list of possible Borel complexity classes of subgroups with a \emph{%
non-Archimedean} Polish cover of abelian groups with a Polish cover: $%
\boldsymbol{\Pi }_{1+\lambda }^{0}$, $\boldsymbol{\Sigma }_{1+\lambda
+1}^{0} $, $D(\boldsymbol{\Pi }_{1+\lambda +n+2}^{0})$, and $\boldsymbol{\Pi 
}_{1+\lambda +n+2}^{0}$ for $\lambda <\omega _{1}$ either zero or a limit
ordinal, and $n<\omega $.
\end{proposition}

The following can be seen as a consequence of \cite[Lemma 3.2 and Theorem 3.3%
]{lupini_complexity_2022} and Proposition \ref%
{Proposition:non-Archimedean-complexity-classes}; see also \cite[Section 5]%
{hjorth_borel_1998}.

\begin{proposition}
\label{Proposition:potential-complexity}Suppose that $G$ is an abelian group
with a Polish cover, and $H$ is a subgroup with a Polish cover of $G$.

\begin{itemize}
\item Let $\Gamma $ be one of the following complexity classes: $\boldsymbol{%
\Pi }_{\alpha }^{0}$, $\boldsymbol{\Sigma }_{\beta }^{0}$, $D(\boldsymbol{%
\Pi }_{\alpha }^{0})$, for $1\leq \alpha ,\beta <\omega _{1}$ and $\beta
\neq 2$. Then $H\in \Gamma \left( G\right) $ if and only if $=_{G/H}$ is
potentially $\Gamma $.

\item $=_{G/H}$ is potentially $\boldsymbol{\Sigma }_{2}^{0}$ if and only if 
$H$ is $D(\boldsymbol{\Pi }_{2}^{0})$ in $G$.

\item If $H$ is a subgroup with a non-Archimedean Polish cover, then $%
=_{G/H} $ is potentially $\boldsymbol{\Sigma }_{2}^{0}$ if and only if $H$
is $\boldsymbol{\Sigma }_{2}^{0}$ in $G$;

\item If $H$ is $\check{D}(\boldsymbol{\Pi }_{\alpha }^{0})$ in $G$ for some 
$1\leq \alpha <\omega _{1}$, then $H\in \boldsymbol{\Pi }_{\alpha
}^{0}\left( G\right) $ or $H\in \boldsymbol{\Sigma }_{\alpha }^{0}\left(
G\right) $.
\end{itemize}
\end{proposition}

As a consequence of Proposition \ref{Proposition:complexity-classes},
Proposition \ref{Proposition:potential-complexity}, Remark \ref%
{Remark:inverse}, and Theorem \ref{Theorem:left-heart-B2}(1) below, one has
the following.

\begin{proposition}
\label{Proposition:complexity-preimage}Suppose that $G,H$ are groups with a
Polish cover, and $f:G\rightarrow H$ is a Borel-definable homomorphism.\ Let 
$\Gamma $ be one of the following Borel complexity classes: $\boldsymbol{\Pi 
}_{\alpha }^{0}$, $\boldsymbol{\Sigma }_{\beta }^{0}$, $D(\boldsymbol{\Pi }%
_{\alpha }^{0})$, $\check{D}(\boldsymbol{\Pi }_{\alpha }^{0})$, $\boldsymbol{%
\Delta }_{\alpha }^{0}$ for $1\leq \alpha ,\beta <\omega _{1}$ and $\beta
\neq 2$. Suppose that $H_{0}$ is a subgroup with a Polish cover of $H$.

\begin{itemize}
\item If $H_{0}$ is $\Gamma $ in $H$, then the subgroup with a Polish cover $%
f^{-1}\left( H_{0}\right) $ of $G$ is $\Gamma $ in $G$. The converse holds
if $f$ is surjective.

\item Suppose that $G,H,H_{0}$ are abelian groups with a non-Archimedean
Polish cover. If $H_{0}$ is $\boldsymbol{\Sigma }_{2}^{0}$ in $H$, then $%
f^{-1}\left( G\right) $ is $\boldsymbol{\Sigma }_{2}^{0}$ in $G$. The
converse holds if $f$ is surjective.
\end{itemize}
\end{proposition}

Recall that $E_{0}$ denotes the $\boldsymbol{\Sigma }_{2}^{0}$ equivalence
relation on the space $\mathcal{C}:=\left\{ 0,1\right\} ^{\omega }$ of
infinite binary sequences obtained by setting $\left( x_{i}\right)
E_{0}\left( y_{i}\right) \Leftrightarrow \exists n\in \omega \forall i\geq n$%
, $x_{i}=y_{i}$, and $E_{\infty }$ denotes the orbit equivalence relation
for the shift action of the free group $\mathbb{F}_{2}$ on $2$ generators on 
$\left\{ 0,1\right\} ^{\mathbb{F}_{2}}$. The $\boldsymbol{\Pi }_{3}^{0}$
equivalence relation $E_{0}^{\omega }$ on $\mathcal{C}^{\omega }$ is defined
by setting $\left( \boldsymbol{x}_{i}\right) E_{0}^{\omega }\left( 
\boldsymbol{y}_{i}\right) \Leftrightarrow \forall i$, $\boldsymbol{x}%
_{i}E_{0}\boldsymbol{y}_{i}$. The $\boldsymbol{\Pi }_{3}^{0}$ equivalence
relation $=^{+}$ on $\mathbb{R}^{\omega }$ is defined by setting $\left(
x_{i}\right) =^{+}\left( y_{i}\right) $ if and only if $\left( x_{i}\right) $
and $\left( y_{i}\right) $ are enumerations of the same countable set of
reals.

\begin{proposition}
\label{Proposition:complexity-equality}Suppose that $G=\hat{G}/N$ is a group
with a Polish cover, and that $H$ is a subgroup of $G$ with a
non-Archimedean Polish cover. Then:

\begin{enumerate}
\item $=_{G/H}$ is smooth if and only if $H$ is $\boldsymbol{\Pi }_{1}^{0}$
in $G$;

\item $=_{G/H}$ is Borel reducible to $E_{0}$ if and only if $=_{G/H}$ is
Borel reducible to $E_{\infty }$ if and only if $H$ is $\boldsymbol{\Sigma }%
_{2}^{0}$ in $G$, and $=_{G/H}$ is Borel bireducible with $E_{0}$ if and
only if $\boldsymbol{\Sigma }_{2}^{0}$ is the complexity class of $H$ in $G$;

\item $=_{G/H}$ is Borel reducible to $E_{0}^{\omega }$ if and only if $%
=_{G/H}$ is Borel reducible to $=^{+}$ if and only if $H$ is $\boldsymbol{%
\Pi }_{3}^{0}$ in $G$, and $=_{G/H}$ is Borel bireducible with $%
E_{0}^{\omega }$ if and only if $\boldsymbol{\Pi }_{3}^{0}$ is the
complexity class of $H$ in $G$.
\end{enumerate}
\end{proposition}

\begin{proof}
Without loss of generality we can assume that $H=\left\{ 0\right\} $.

(1) We have that $\left\{ 0\right\} $ is $\boldsymbol{\Pi }_{1}^{0}$ in $G$
if and only if $N$ is a closed subgroup of $\hat{G}$, which is equivalent to
the assertion that $=_{G}$ is smooth; see \cite[page 574]{solecki_coset_2009}%
.

(2) By \cite[Theorem 12.5.7 and Theorem 7.3.8]{gao_invariant_2009}, $\left\{
0\right\} $ is $\boldsymbol{\Sigma }_{2}^{0}$ in $G$ if and only if $%
=_{G}\leq _{B}E_{\infty }$, which holds if and only if $=_{G}\leq _{B}E_{0}$
by \cite[Theorem 6.1]{ding_non-archimedean_2017}. Furthermore, by Item (1),
Pettis' Theorem \cite[Theorem 9.9]{kechris_classical_1995}, and the
Glimm--Effros dichotomy \cite{harrington_glimm-effros_1990}, we have that $%
\left\{ 0\right\} $ is not $\boldsymbol{\Pi }_{2}^{0}$ if and only if $%
\left\{ 0\right\} $ is not $\boldsymbol{\Pi }_{1}^{0}$ if and only if $N$ is
not a closed subgroup of $\hat{G}$, if and only if $E_{0}\leq _{B}=_{G}$.

(3) By \cite[Corollary 6.3]{ding_non-archimedean_2017} and Proposition \ref%
{Proposition:potential-complexity} and we have that $\left\{ 0\right\} $ is
not $\boldsymbol{\Sigma }_{3}^{0}$ if and only if $\left\{ 0\right\} $ is
not $\boldsymbol{\Sigma }_{2}^{0}$ if and only if $E_{0}^{\omega }\leq
_{B}=_{G}$. By \cite[Corollary 6.11]{allison_non-archimedean_2020}, $\left\{
0\right\} $ is $\boldsymbol{\Pi }_{3}^{0}$ if and only if $=_{G}\leq
_{B}E_{0}^{\omega }$, and by \cite[Theorem 12.5.5]{gao_invariant_2009}, $%
\left\{ 0\right\} $ is $\boldsymbol{\Pi }_{3}^{0}$ if and only if $=_{G}\leq
_{B}=^{+}$.
\end{proof}

\subsection{The Solecki subgroups\label{Section:solecki}}

Every abelian group with a Polish cover admits a canonical sequence of
subgroups indexed by countable ordinals. As these were originally described
by Solecki in \cite{solecki_polish_1999}, we call them \emph{Solecki
subgroups}. They have also been considered in \cite{solecki_coset_2009,
farah_borel_2006}.

Suppose that $G=\hat{G}/N$ is an abelian group with a Polish cover. Then 
\cite[Lemma 2.3]{solecki_polish_1999} implies that $G$ has a smallest $%
\boldsymbol{\Pi }_{3}^{0}$ subgroup, which we denote by $s_{1}\left(
G\right) =s_{1}^{N}(\hat{G})/N$. One can explicitly describe $s_{1}^{N}(\hat{%
G})$ as the subgroup of $\hat{G}$ defined by%
\begin{equation*}
\bigcap_{V}\bigcup_{z\in N}\overline{z+V}^{G}
\end{equation*}%
where $V$ ranges among the open zero neighborhoods in $N$ and $\overline{z+V}%
^{\hat{G}}$ is the closure of $z+V$ inside of $\hat{G}$. It is proved in 
\cite[Lemma 2.3]{solecki_polish_1999} that $s_{1}(G)$ satisfies the
following properties:

\begin{itemize}
\item $s_{1}\left( G\right) $ is a subgroup with a Polish cover;

\item $\left\{ 0\right\} $\emph{\ }is dense in $s_{1}\left( G\right) $;

\item a basis of zero neighborhoods in $s_{1}^{N}(\hat{G})$ consists of sets
of the form $\overline{W}^{\hat{G}}\cap s_{1}^{N}(\hat{G})$ where $W$ is an
open zero neighborhood in $N$;

\item if $A\subseteq \hat{G}$ is $\boldsymbol{\Pi }_{3}^{0}$ and contains $N$%
, then $A\cap s_{1}^{N}(\hat{G})$ is comeager in the Polish group topology
of $s_{1}^{N}(\hat{G})$.
\end{itemize}

It follows that if $H$ is a $\boldsymbol{\Pi }_{3}^{0}$ subgroup with a
Polish cover of $G$, then $s_{1}\left( G\right) \subseteq \overline{\left\{
0\right\} }^{H}\subseteq H$. We recall the following characterization of $%
s_{1}\left( G\right) $ from \cite[Lemma 4.2]{lupini_complexity_2022}.

\begin{lemma}
\label{Lemma:characterize-solecki}Suppose that $G=\hat{G}/N$ is an abelian
group with a Polish cover.\ Let $H=\hat{H}/N$ be a subgroup with a Polish
cover of $G$ such that:

\begin{enumerate}
\item $\left\{ 0\right\} $ is dense in $H$;

\item for every open neighborhood $V$ of zero in $N$, $\overline{V}^{\hat{G}%
}\cap \hat{H}$ contains an open neighborhood of zero in $\hat{H}$.
\end{enumerate}

If $A\subseteq \hat{G}$ is $\boldsymbol{\Pi }_{3}^{0}$ and contains $N$,
then $A\cap \hat{H}$ is comeager in $\hat{H}$. In particular, $H\subseteq
s_{1}\left( G\right) $. If $H$ is furthermore $\boldsymbol{\Pi }_{3}^{0}$,
then $H=s_{1}\left( G\right) $.
\end{lemma}

The sequence of Solecki subgroups $s_{\alpha }\left( G\right) $ for $\alpha
<\omega _{1}$ of the group with a Polish cover $G$ is defined recursively by
setting:

\begin{itemize}
\item $s_{0}\left( G\right) =\overline{\left\{ 0\right\} }^{G}$;

\item $s_{\alpha +1}\left( G\right) =s_{1}\left( s_{\alpha }\left( G\right)
\right) $ for $\alpha <\omega _{1}$;

\item $s_{\lambda }\left( G\right) =\bigcap_{\beta <\lambda }s_{\beta
}\left( G\right) $ for a limit ordinal $\lambda <\omega _{1}$.
\end{itemize}

We also let $s_{\alpha }^{N}(\hat{G})$ be the Polishable subgroup of $\hat{G}
$ such that $s_{\alpha }\left( G\right) =s_{\alpha }^{N}(\hat{G})/N$. One
can prove by induction on $\alpha <\omega _{1}$ that $\left\{ 0\right\} $ is
dense in $s_{\alpha }\left( G\right) $ for every $\alpha <\omega _{1}$, and
if $\left\{ 0\right\} $ is a subgroup with a non-Archimedean Polish cover of 
$G$, then $s_{\alpha }\left( G\right) $ is a subgroup with a non-Archimedean
Polish cover for every $1\leq \alpha <\omega _{1}$; see \cite[Section 4]%
{lupini_complexity_2022}. It is proved in \cite[Theorem 2.1]%
{solecki_polish_1999} that there exists $\alpha <\omega _{1}$ such that $%
s_{\alpha }\left( G\right) =\left\{ 0\right\} $. We call the least countable
ordinal $\alpha $ such that $s_{\alpha }\left( G\right) =\left\{ 0\right\} $
the \emph{Solecki rank} of $G$.

The following is an immediate consequence of \cite[Theorem 5.4]%
{lupini_complexity_2022}.

\begin{theorem}
\label{Theorem:characterize-Solecki}Suppose that $G$ is an abelian group
with a Polish cover, and $\alpha <\omega _{1}$. Then $s_{\alpha }\left(
G\right) $ is the smallest $\boldsymbol{\Pi }_{1+\alpha +1}^{0}$ subgroup of 
$G$.
\end{theorem}

As a consequence of Proposition \ref{Proposition:complexity-preimage} and
Theorem \ref{Theorem:characterize-Solecki} one obtains the following.

\begin{theorem}
Fix $\alpha <\omega _{1}$. If $f:G\rightarrow H$ is a Borel-definable
homomorphism between groups with a Polish cover, then $f$ maps $s_{\alpha
}\left( G\right) $ to $s_{\alpha }\left( H\right) $. Thus, $G\mapsto
s_{\alpha }\left( G\right) $ is a \emph{subfunctor of the identity }on the
category of abelian groups with a (non-Archimedean) Polish cover.
\end{theorem}

We also have the following consequence of \cite[Theorem 6.1]%
{lupini_complexity_2022}.

\begin{theorem}
\label{Theorem:complexity-Solecki}Suppose that $G=\hat{G}/N$ is an abelian
group with a Polish cover. Let $\alpha =\lambda +n$ be the Solecki rank of $%
G $, where $\lambda <\omega _{1}$ is either zero or a limit ordinal and $%
n<\omega $.

\begin{enumerate}
\item Suppose that $n=0$.\ Then $\boldsymbol{\Pi }_{1+\lambda }^{0}$ is the
complexity class of $\left\{ 0\right\} $ in $G$;

\item Suppose that $n\geq 1$. Then:

\begin{enumerate}
\item if $\left\{ 0\right\} \in \boldsymbol{\Pi }_{3}^{0}\left( s_{\lambda
+n-1}\left( G\right) \right) $ and $\left\{ 0\right\} \notin D(\boldsymbol{%
\Pi }_{2}^{0})\left( s_{\lambda +n-1}\left( G\right) \right) $, then $%
\boldsymbol{\Pi }_{1+\lambda +n+1}^{0}$ is the complexity class of $\left\{
0\right\} $ in $G$;

\item if $n\geq 2$ and $\left\{ 0\right\} \in D(\boldsymbol{\Pi }%
_{2}^{0})(s_{\lambda +n-1}\left( G\right) )$, then $D(\boldsymbol{\Pi }%
_{1+\lambda +n}^{0})$ is the complexity class of $\left\{ 0\right\} $ in $G$;

\item if $n=1$, $\left\{ 0\right\} \in D(\boldsymbol{\Pi }_{2}^{0})\left(
s_{\lambda }\left( G\right) \right) $, and $\left\{ 0\right\} \notin 
\boldsymbol{\Sigma }_{2}^{0}\left( s_{\lambda }\left( G\right) \right) $,
then $D(\boldsymbol{\Pi }_{1+\lambda +1}^{0})$ is the complexity class of $%
\left\{ 0\right\} $ in $G$;

\item if $n=1$ and $\left\{ 0\right\} \in \boldsymbol{\Sigma }_{2}^{0}\left(
s_{\lambda }\left( G\right) \right) $, then $\boldsymbol{\Sigma }_{1+\lambda
+1}^{0}$ is the complexity class of $\left\{ 0\right\} $ in $G$.
\end{enumerate}
\end{enumerate}

Furthermore, if $\left\{ 0\right\} $ is a subgroup of $G$ with a
non-Archimedean Polish cover, then the case \emph{(2c)} is excluded.
\end{theorem}

\subsection{Ulm subgroups\label{Section:Ulm}}

Suppose that $G=\hat{G}/N$ is an abelian group with a Polish cover. The
first Ulm subgroup $u_{1}\left( G\right) =G^{1}$ is the subgroup $%
\bigcap_{n>0}nG$ of $G$. One then defines the sequence $u_{\alpha }\left(
G\right) $ of Ulm subgroups by recursion on $\alpha <\omega _{1}$ by setting:

\begin{enumerate}
\item $u_{0}\left( G\right) =G$;

\item $u_{\alpha +1}\left( G\right) =u_{1}\left( u_{\alpha }\left( G\right)
\right) $;

\item $u_{\lambda }\left( G\right) =\bigcap_{\alpha <\lambda }u_{\alpha
}\left( G\right) $ for $\lambda $ limit.
\end{enumerate}

It follows from Lemma \ref{Lemma:intersection--Polishable} by induction on $%
\alpha <\omega _{1}$ that $u_{\alpha }\left( G\right) $ is a subgroup of $G$
with a Polish cover, which is non-Archimedean if $G$ is a abelian group with
a non-Archimedean Polish cover; see Definition \ref{Definition:non-A}.

We let $D\left( G\right) $ be the largest divisible subgroup of $G$. A group 
$G$ is \emph{reduced }if $D\left( G\right) =0$. Notice that $D\left(
G\right) $ is a Polishable subgroup of $G$, which is non-Archimedean when $G$
is a abelian group with a non-Archimedean Polish cover. Indeed, $D\left(
G\right) =\hat{D}\left( G\right) /N$, where $\hat{D}\left( G\right) $ is the
image of the Polish group%
\begin{equation*}
H=\left\{ \left( g_{n},r_{n}\right) _{n\in \omega }:\forall n\in \omega
,g_{n}-\left( n+1\right) g_{n+1}=r_{n}\right\} \subseteq \prod_{n\in \omega }%
\hat{G}\oplus N
\end{equation*}%
under a continuous group homomorphism. The same proof as \cite[Theorem 4.1(i)%
]{farah_borel_2006} gives the following.

\begin{proposition}
\label{Lemma:Ulm}Suppose that $G=\hat{G}/N$ is a group with a Polish cover.
Then there exists $\alpha <\omega _{1}$ such that $u_{\alpha }\left(
G\right) =D\left( G\right) $.
\end{proposition}

\begin{proof}
After replacing $G$ with $G/D\left( G\right) $, we can assume that $D\left(
G\right) =\left\{ 0\right\} $. Set $B:=\hat{G}\setminus N$. Define the Borel
relation $\prec $ on the standard Borel space $B^{<\omega }$ of finite
sequences in $B$ by setting%
\begin{equation*}
\left( g_{0},\ldots ,g_{n}\right) \prec \left( h_{0},\ldots ,h_{m}\right)
\Leftrightarrow m<n\text{ and }\forall i\leq m\text{, }h_{i}=g_{i}\text{,
and }\forall i\leq n\text{, }\left( i+1\right) g_{i+1}\equiv g_{i}\mathrm{\ 
\mathrm{mod}}\ N\text{.}
\end{equation*}%
Then $\prec $ is well-founded, and hence its rank $\rho \left( \prec \right) 
$ is a countable ordinal \cite[Theorem 31.1]{kechris_classical_1995}. For $%
A\subseteq B^{<\omega }$ define 
\begin{equation*}
A^{\prime }=\left\{ s\in B^{<\omega }:\exists t\in A,t\prec s\right\} \text{.%
}
\end{equation*}%
Set then recursively $A_{0}=B^{<\omega }$, $A_{\alpha +1}=A_{\alpha
}^{\prime }$, and $A_{\lambda }=\bigcap_{\alpha <\lambda }A_{\alpha }$ for $%
\lambda $ limit. Then it is easily proved by induction that, for every
ordinal $\alpha $ and $s=\left( g_{0},\ldots ,g_{n}\right) \in B^{<\omega }$:

\begin{itemize}
\item $\rho _{\prec }\left( s\right) \geq \alpha $ if and only if $s\in
A_{\alpha }$;

\item $s\in A_{\omega \alpha }\Leftrightarrow g_{n}+N\in u_{\alpha }\left(
G\right) $.
\end{itemize}

Thus, if $\lambda $ is a countable ordinal such that $\rho \left( \prec
\right) \leq \omega \lambda $, then we have that $\rho _{\prec }\left(
s\right) <\omega \lambda $ for every $s\in B^{<\omega }$, and hence $%
A_{\omega \lambda }=\varnothing $ and $u_{\lambda }\left( G\right) =\left\{
0\right\} $.
\end{proof}

Suppose that $G$ is an abelian group with a Polish cover. The \emph{Ulm rank}
of $G$ the least $\alpha <\omega _{1}$ such that $u_{\alpha }\left( G\right)
=D\left( G\right) $. Notice that such an $\alpha $ exists by Proposition \ref%
{Lemma:Ulm}.

\subsection{Polish modules\label{Subsection:modules}}

In this section, we observe how all the results that we have obtained so far
apply more generally in the context of Polish $G$-modules.

Suppose that $G$ is a (multiplicatively denoted) Polish group, and $R$ is a
Polish ring. We say that $A$ is a Polish $G$-module if $A$ is an abelian
Polish group endowed with a continuous action $G\curvearrowright A$ by
automorphism of $A$, denoted by $\left( g,a\right) \mapsto g\cdot a$ \cite[%
Section 3]{moore_group_1976}. We say that $A$ is a Polish $R$-module if it
is an abelian Polish group that is also an $R$-module, such that the scalar
multiplication operation is continuous. We now recall some automatic
continuity results for modules; see also \cite[Proposition 11]%
{moore_group_1976}.

The following lemma guarantees that a separately continuous function is
continuous on a large set; see \cite[Theorem 8.51]{kechris_classical_1995}.

\begin{lemma}
\label{Lemma:separately-continuous}Let $X,Y,Z$ be Polish spaces and $%
f:X\times Y\rightarrow Z$ be a function that is separately continuous. Then
there exists a dense $G_{\delta }$ set $C\subseteq X\times Y$ such that, for
all $y\in Y$, $C^{y}=\left\{ x\in X:\left( x,y\right) \in C\right\} $ is a
dense $G_{\delta }$ in $X$, and $f$ is continuous at every point of $C$.
\end{lemma}

As an application of the automatic continuity of Borel group homomorphism
between Polish groups \cite[Theorem 9.10]{kechris_classical_1995} and Lemma %
\ref{Lemma:separately-continuous} one has the following.

\begin{lemma}
\label{Lemma:continuous-G-module}Suppose that $G$ is a Polish group, and $A$
is an abelian Polish group. If an action $G\curvearrowright A$ by
automorphisms of $A$ is Borel separately in each variable when seen as a
function $G\times A\rightarrow A$, then it is continuous.
\end{lemma}

\begin{proof}
For $g_{0}\in G$, the map $A\rightarrow A$, $a\mapsto g_{0}\cdot a$ is a
Borel automorphism of $A$, and hence continuous \cite[Therem 9.10]%
{kechris_classical_1995}. For $a_{0}\in A$ we have that the map $\varepsilon
_{a_{0}}:G\rightarrow A$, $g\mapsto g\cdot a_{0}$ is Borel. Thus there
exists a dense $G_{\delta }$ subset $D$ of $G$ such that $\varepsilon
_{a_{0}}|_{D}$ is continuous \cite[Theorem 8.38]{kechris_classical_1995}. We
now prove that $\varepsilon _{a_{0}}$ is continuous at an arbitrary $%
g_{\infty }\in G$. Suppose that $\left( g_{n}\right) _{n\in \mathbb{N}}$ is
a sequence converging to $g_{\infty }$. Consider $h\in Dg_{\infty }^{-1}\cap
\bigcap_{n\in \mathbb{N}}Dg_{n}^{-1}$ Thus, $hg_{n}\in D$ for every $n\in 
\mathbb{N}\cup \left\{ \infty \right\} $, and $hg_{n}\rightarrow hg_{\infty
} $. Thus, $\left( hg_{n}\cdot a_{0}\right) _{n\in \mathbb{N}}$ converges to 
$hg_{\infty }\cdot a_{0}$. Since the function $a\mapsto h^{-1}\cdot a$ is a
continuous automorphism of $A$, we have that $\left( g_{n}\cdot a_{0}\right)
_{n\in \mathbb{N}}$ converges to $g_{\infty }\cdot a_{0}$. This shows that $%
\varepsilon _{a_{0}}$ is continuous at $g_{\infty }$.

Finally, by the above and Lemma \ref{Lemma:separately-continuous} there
exists a dense $G_{\delta }$ subset $C$ of $G\times A$ such that, for every $%
a\in A$, $C^{a}$ is a dense $G_{\delta }$ subset of $G$ and the function $%
\left( g,a\right) \mapsto g\cdot a$ is continuous at every point of $C$. Fix
now $\left( g_{0},a_{0}\right) \in G\times A$. Let $h_{0}\in G$ such that $%
m:\left( g,a\right) \mapsto g\cdot a$ is continuous at $\left(
h_{0},a_{0}\right) $. Then we can write the function $m$ as $\left(
g,a\right) \mapsto \left( g_{0}h_{0}^{-1}\right) \cdot \left(
h_{0}g_{0}^{-1}g\cdot a\right) $. This realizes $m$ as a composition of a
function that is continuous at $\left( g_{0},a_{0}\right) $ with a
continuous function.\ Thus, $m$ is continuous at $\left( g_{0},a_{0}\right) $%
. Since $\left( g_{0},a_{0}\right) $ is an arbitrary element of $G\times A$,
we have that $m$ is continuous.
\end{proof}

\begin{lemma}
\label{Lemma:continuous-R-module}Suppose that $R$ is a Polish ring, and $A$
is an $R$-module and an abelian Polish group. If the scalar multiplication
operation $R\times A\rightarrow A$, $\left( \lambda ,x\right) \mapsto
\lambda \cdot x$ is Borel separately in each variable, then it is continuous.
\end{lemma}

\begin{proof}
As in the proof of Lemma \ref{Lemma:continuous-G-module}, for every $\lambda
_{0}\in R$, the group homomorphism $A\rightarrow A$, $x\mapsto \lambda x$ of 
$A$ is Borel, and hence continuous. For the same reason, for every $a_{0}\in
A$, the map $R\rightarrow A$, $\lambda \mapsto \lambda a_{0}$ is a Borel
group homomorphism, and hence continuous. Thus, by Lemma \ref%
{Lemma:separately-continuous} there exists a dense $G_{\delta }$ subset $C$
of $R\times A$ such that for every $a\in a$, $C^{a}$ is a dense $G_{\delta }$
subset of $R$ and the function $m:\left( \lambda ,a\right) \mapsto \lambda
\cdot a$ is continuous at every point of $C$. Fix $\left( \lambda
_{0},a_{0}\right) \in R\times A$ and pick $\mu _{0}\in A$ such that the
function $m$ is continuous at $\left( \mu _{0},a_{0}\right) $. Then we can
write $m$ as the function $\left( \lambda ,a\right) \mapsto \left( \mu
_{0}-\lambda _{0}+\lambda \right) \cdot a+\left( \mu _{0}-\lambda
_{0}\right) \cdot a$. This witnesses that $m$ is continuous at $\left( \mu
_{0},a_{0}\right) $. Being $\left( \mu _{0},a_{0}\right) $ an arbitrary
element of $R\times A$, we have that $m$ is continuous.
\end{proof}

Suppose that $R$ is a Polish group or a Polish ring. The notions of
Polishable $R$-submodule of a Polish module, $R$-module with a Polish cover, 
$R$-submodule with a Polish cover, and Borel-definable, continuously
definable, and liftable $R$-homomorphism between $R$-modules with a Polish
cover are defined as in the group case. It follows easily from Lemma \ref%
{Lemma:continuous-G-module} and Lemma \ref{Lemma:continuous-R-module} that
all the results that we have obtained so far about groups with a Polish
cover apply more generally to Polish $R$-modules. Furthermore, it is not
difficult to see that if $M$ is a $R$-module with a Polish cover, then its
Solecki subgroups are in fact $R$-submodules with a Polish cover.

\section{Better lifts\label{Section:better-lifts}}

In this section, we prove that under certain circumstances a continuously
definable group homomorphism between abelian groups with a Polish cover
admits a lift that satisfies additional properties besides being continuous.

\subsection{Approximately additive lifts}

We begin with considering the existence of continuous lifts with additional
properties for continuously definable group homomorphisms between groups
with a Polish cover.

\begin{definition}
Suppose that $G$ is an abelian Polish group and $H=\hat{H}/M$ is a group
with a Polish cover. A Borel lift $f:G\rightarrow \hat{H}$ for a group
homomorphism $G\rightarrow H$ is \emph{approximately additive }if $f\left(
0\right) =0$ and the function $\delta f:G\times G\rightarrow M$, $\left(
x,y\right) \mapsto f\left( y\right) -f\left( x+y\right) +f\left( x\right) $
is continuous at $\left( 0,0\right) $.
\end{definition}

A similar proof as \cite[Lemma 4.9]{bergfalk_definable_2020} gives the
following.

\begin{lemma}
\label{Lemma:approximately-additive-1step}Suppose that $G$ is an abelian
Polish group, $H=\hat{H}/M$ is an abelian group with a Polish cover, and $%
f:G\rightarrow \hat{H}$ is a Borel lift of a group homomorphism $%
G\rightarrow H$ that is continuous on a zero neighborhood $G_{0}$ of $G$.
Let $M_{1}$ be a zero neighborhood in $M$ such that $M_{1}=\overline{M}_{1}^{%
\hat{H}}\cap M$, where $\overline{M}_{1}^{\hat{H}}$ is the closure of $M_{1}$
in $\hat{H}$. Then there exist $x_{0},y_{0}\in G_{0}$, and a zero
neighborhood $G_{1}$ in $G$ contained in $G_{0}$ such that:

\begin{itemize}
\item for $x,y\in G_{1}$, 
\begin{equation*}
f\left( x+x_{0}\right) +f\left( y+y_{0}\right) -f\left(
x+y+x_{0}+y_{0}\right) \in M_{1}\text{;}
\end{equation*}

\item if $g:G\rightarrow \hat{H}$ is defined by%
\begin{equation*}
g\left( z\right) :=f\left( x_{0}+y_{0}+z\right) -f\left( x_{0}+y_{0}\right) 
\text{,}
\end{equation*}%
then, for every $x,y\in G_{1}$,%
\begin{equation*}
\delta g\left( x,y\right) \in M_{1}+M_{1}+M_{1}+M_{1}\text{.}
\end{equation*}
\end{itemize}
\end{lemma}

\begin{proof}
Since $M_{1}$ is non-meager in $M$, there exists $m\in M$ such that 
\begin{equation*}
A:=\{\left( x,y\right) \in G_{0}\times G_{0}:\delta f\left( x,y\right) \in
m+M_{1}\}
\end{equation*}%
is non-meager in $G_{0}\times G_{0}$. After replacing $f$ with $z\mapsto
f\left( z\right) -m$, we can assume that $m=0$. Since $f:G_{0}\rightarrow 
\hat{H}$ is continuous and $M_{1}=\overline{M}_{1}^{\hat{H}}\cap M$, we have
that $A\subseteq G_{0}\times G_{0}$ is closed.\ Thus, $A$ is somewhere
dense, and there exists a zero neighborhood $G_{1}$ in $G_{0}$ and $%
x_{0},y_{0}\in G_{0}$ such that $\left( x_{0},y_{0}\right) +\left(
G_{1}\times G_{1}\right) \subseteq A$. Thus, for $x,y\in G_{1}$ we have that%
\begin{equation}
f\left( x+x_{0}\right) +f\left( y+y_{0}\right) -f\left(
x+y+x_{0}+y_{0}\right) \in M_{1}\text{.\label{Identita-sbuccia}}
\end{equation}%
Define thus $g:G\rightarrow \hat{H}$ by%
\begin{equation*}
g\left( z\right) :=f\left( x_{0}+y_{0}+z\right) -f\left( x_{0}+y_{0}\right) 
\text{.}
\end{equation*}%
Then we have that, for $x,y\in G_{1}$,%
\begin{eqnarray*}
\delta g\left( x,y\right) &=&g\left( x+y\right) -g\left( x\right) -g\left(
y\right) \\
&=&f\left( x_{0}+y_{0}+x+y\right) -f\left( x_{0}+y_{0}+x\right) -f\left(
x_{0}+y_{0}+y\right) +f\left( x_{0}+y_{0}\right) \text{.}
\end{eqnarray*}%
By \eqref{Identita-sbuccia}, we have%
\begin{equation*}
f\left( x_{0}+y_{0}+x+y\right) -f\left( x_{0}+y_{0}+x\right) -f\left(
x_{0}+y_{0}+y\right) +f\left( x_{0}+y_{0}\right) \in M_{1}+M_{1}+M_{1}+M_{1}
\end{equation*}%
This concludes the proof.
\end{proof}

One can infer from Lemma \ref{Lemma:approximately-additive-1step}, as in the
proof of \cite[Theorem 4.5]{bergfalk_definable_2020}, the following lemma.

\begin{proposition}
\label{Proposition:approximately-additive}Suppose that $G\ $is an abelian
Polish group, $H=\hat{H}/M$ is an abelian group with a Polish cover, and $%
\varphi :G\rightarrow H$ is a locally continuously definable group
homomorphism. Suppose that $M$ has a basis of zero neighborhoods that are
closed in the subspace topology inherited from $H$. Then $\varphi $ has an 
\emph{approximately additive }locally continuous Borel lift.
\end{proposition}

\begin{proof}
By assumption, we have that $M$ has a basis $\left( M_{k}\right) _{k\in
\omega }$ of zero neighborhoods such that $\overline{M}_{k}^{\hat{H}}\cap
M=M_{k}$ for every $k\in \omega $. Let $(\hat{H}_{k})$ be a basis of zero
neighborhoods in $\hat{H}$ such that $\hat{H}_{0}=\hat{H}$. Without loss of
generality, we can assume that $M_{0}=M$ and $M_{k}\subseteq \hat{H}_{k}$
for $k\in \omega $. Since $\varphi $ is locally continuously definable, it
has a Borel lift $f:G\rightarrow \hat{H}$ that is continuous on a zero
neighborhood $U$ in $G$. Let also $d_{G}$ be a compatible complete invariant
metric on $G$ such that $\left\{ z\in G_{0}:d_{G}\left( z,0\right) \leq
2\right\} \subseteq U$. Let $G_{0}$ be the zero neighborhood $\left\{ z\in
G_{0}:d_{G}\left( z,0\right) \leq 1/4\right\} $ in $G$, and set $f_{0}:=f$.

Applying Lemma \ref{Lemma:approximately-additive-1step}, one can define by
recursion on $k\in \omega $:

\begin{itemize}
\item a zero neighborhood $G_{k+1}$ of $G$ contained in $\left\{ x\in
G_{k}:d_{G}\left( x,0\right) \leq 2^{-\left( k+2\right) }\right\} $;

\item elements $x_{k},y_{k}\in G_{k}$;

\item a Borel function $f_{k}:G\rightarrow \hat{H}$ ,
\end{itemize}

such that, for every $k\in \omega $:

\begin{enumerate}
\item for every $z\in G$,%
\begin{equation*}
f_{k+1}\left( z\right) =f_{k}\left( x_{k}+y_{k}+z\right) -f_{k}\left(
x_{k}+y_{k}\right) \text{;}
\end{equation*}

\item for every $x,y\in G_{k}$, 
\begin{equation*}
f_{k}\left( x+y+x_{k}+y_{k}\right) \equiv f_{k}\left( x+x_{k}\right)
+f_{k}\left( y+y_{k}\right) \mathrm{\ \mathrm{mod}}\ M_{k+1}\text{;}
\end{equation*}

\item for every $x,y\in G_{k+1}$,%
\begin{equation*}
\delta f_{k+1}\left( x,y\right) \in M_{k+1}\text{.}
\end{equation*}
\end{enumerate}

Indeed, suppose that $k\geq 0$, and $G_{i+1}$, $f_{i+1}$, and $%
x_{i},y_{i}\in G_{i}$ have been defined for $i<k$. We apply Lemma \ref%
{Lemma:approximately-additive-1step} to obtain elements $x_{k},y_{k}\in
G_{k} $ and a zero neighborhood $G_{k+1}$ in $G$ contained in $\left\{ x\in
G_{k}:d_{G}\left( x,0\right) \leq 2^{-\left( k+2\right) }\right\} $ such
that, setting%
\begin{equation*}
f_{k+1}\left( z\right) :=f_{k}\left( x_{k}+y_{k}+z\right) -f_{k}\left(
x_{k}+y_{k}\right)
\end{equation*}%
for $z\in G$, we have that%
\begin{equation*}
f_{k}\left( x+y+x_{k}+y_{k}\right) \equiv f_{k}\left( x+x_{k}\right)
+f_{k}\left( y+y_{k}\right) \mathrm{\ \mathrm{mod}}\ M_{k+1}
\end{equation*}%
and%
\begin{equation*}
\delta f_{k+1}\left( x,y\right) \in M_{k+1}\text{.}
\end{equation*}%
For $k\in \omega $, set%
\begin{equation*}
z_{k}:=\left( x_{0}+y_{0}\right) +\cdots +\left( x_{k}+y_{k}\right) \text{.}
\end{equation*}%
We prove by induction on $k\in \omega $ that, for every $z\in G$,%
\begin{equation*}
f_{k+1}\left( z\right) =f\left( z_{k}+z\right) -f\left( z_{k}\right) \text{.}
\end{equation*}%
For $k=0$ this holds by definition. Suppose that it holds for $k$. Then we
have that, by definition and the inductive hypothesis,%
\begin{eqnarray*}
f_{k+2}\left( z\right) &=&f_{k+1}\left( x_{k+1}+y_{k+1}+z\right)
-f_{k+1}\left( x_{k+1}+y_{k+1}\right) \\
&=&\left( f\left( z_{k}+x_{k+1}+y_{k+1}+z\right) -f\left( z_{k}\right)
\right) -\left( f\left( z_{k}+x_{k+1}+y_{k+1}\right) -f\left( z_{k}\right)
\right) \\
&=&f\left( z_{k+1}+z\right) -f\left( z_{k+1}\right) \text{.}
\end{eqnarray*}%
Notice that 
\begin{equation*}
d_{G}\left( x_{i}+y_{i},0\right) \leq 2^{-\left( i+1\right) }
\end{equation*}%
for every $i<\omega $, and hence the sequence $\left( z_{k}\right) _{k\in
\omega }$ converges to some $z_{\infty }\in G$ such that $d_{G}\left(
z_{\infty },0\right) \leq 1$.

Set $W:=\left\{ z\in G:d_{G}\left( z,0\right) \leq 1\right\} $. Notice that,
for every $i\in \omega $ and $z\in W$, $z_{i}+z\in U$ and $z_{\infty }+z\in
U $. Define, for $z\in W$,%
\begin{equation*}
g\left( z\right) :=f\left( z+z_{\infty }\right) -f\left( z_{\infty }\right) =%
\mathrm{lim}_{i\rightarrow \infty }\left( f\left( z+z_{i}\right) -f\left(
z_{i}\right) \right) =\mathrm{lim}_{i\rightarrow \infty }f_{i}\left(
z\right) \text{.}
\end{equation*}%
Since the family of functions $\left( f_{i}\right) _{i\in \omega }$ is
uniformly equicontinuous on $W$, we have that the function $g:W\rightarrow 
\hat{H}$ is continuous. Furthermore, $g$ satisfies $g\left( x\right)
+M=\varphi \left( x\right) $ for $x\in W$ and, for every $k\in \omega $ and $%
x,y\in G_{k}$,%
\begin{equation*}
\delta g\left( x,y\right) \in M_{k}\text{.}
\end{equation*}%
Finally, one can extend $g$ to a Borel lift $g:G\rightarrow \hat{H}$ for $%
\varphi $ using Lemma \ref{Lemma:locally-continuously}.
\end{proof}

\begin{remark}
\label{Remark:approximately-additive-analytic}A similar proof as Proposition %
\ref{Proposition:approximately-additive} gives the following result: suppose
that $G\ $is an abelian real Lie group, $H=\hat{H}/M$ is an abelian group
with a real Lie cover, and $\varphi :G\rightarrow H$ is a group homomorphism
that has a Borel lift $G\rightarrow \hat{H}$ that is analytic on an open
zero neighborhood in $G$. Then $\varphi $ has an \emph{approximately
additive }Borel lift that is analytic on an open zero neighborhood in $G$.
\end{remark}

\subsection{Approximately $R$-linear lifts}

A \emph{non-Archimedean Polish ring} is a Polish ring that has a basis of
(open) zero neighborhoods consisting of subrings. Let $R$ be a
non-Archimedean Polish ring. A subset $A$ of $R$ is \emph{bounded} if for
every zero neighborhood $U$ in $R$ there exists a zero neighborhood $V$ of $%
R $ such that $V\cdot A\subseteq U$ \cite{weiss_boundedness_1956}. We say
that $R$ is \emph{locally bounded} if it has a bounded zero neighborhood. A
Polish $R$-module $X$ is \emph{non-Archimedean} if for every zero
neighborhood $U$ of $X$ there exists an open subring $O$ of $R$ and an open $%
O$-submodule $V$ of $X$ contained in $U$.

\begin{definition}
\label{Definition:approximately-linear-lift}Suppose that $R$ is a
non-Archimedean Polish ring and that $X=\hat{X}/N$ and $Y=\hat{Y}/M$ are $R$%
-modules with a non-Archimedean Polish cover. An \emph{approximately }$R$-%
\emph{linear continuous lift }of an $R$-homomorphism $\varphi :X\rightarrow
Y $ is an approximately additive continuous lift $f:\hat{X}\rightarrow \hat{Y%
}$ such that for every zero neighborhood $U$ of $\hat{Y}$ there exists an
open subring $O$ of $R$ and an open $O$-submodule $W$ of $\hat{X}$ such that 
$f\left( \lambda x\right) \equiv \lambda f\left( x\right) \ \mathrm{\mathrm{%
mod}}U$ for every $\lambda \in O$ and $x\in W$.
\end{definition}

\begin{lemma}
\label{Lemma:approximately-linear-lift}Suppose that $R$ is a non-Archimedean
Polish ring, $X$ is non-Archimedean Polish $R$-module, and $Y=\hat{Y}/M$ is
an $R$-module with a non-Archimedean Polish cover. Let $g:X\rightarrow \hat{Y%
}$ be a continuous lift for an $R$-homomorphism $X\rightarrow Y$. Let $M_{1}$
be an open subgroup of $M$ that is closed in the subspace topology inherited
from $\hat{Y}$. Suppose that $X_{1}$ is an open subgroup of $X$ such that,
for $x,y\in X_{1}$,%
\begin{equation*}
g\left( x+y\right) \equiv g\left( x\right) +g\left( y\right) \ \mathrm{mod}%
M_{1}\text{.}
\end{equation*}%
Then there exist an open subring $O$ of $R$ and an open $O$-submodule $%
X_{1}^{\prime }\subseteq X$ such that for every $x\in X_{1}^{\prime }$ and $%
\lambda \in O$,%
\begin{equation*}
g\left( \lambda x\right) \equiv \lambda g\left( x\right) \ \mathrm{mod}M_{1}%
\text{.}
\end{equation*}%
Furthermore, for every bounded open subring $S$ of $R$ there exists an open $%
S$-submodule $X_{1}^{\prime \prime }\subseteq X_{1}$ such that for every $%
x\in X_{1}^{\prime \prime }$ and $\lambda \in S$,%
\begin{equation*}
g\left( \lambda x\right) \equiv \lambda g\left( x\right) \ \mathrm{mod}M_{1}%
\text{.}
\end{equation*}
\end{lemma}

\begin{proof}
Consider the Borel function%
\begin{equation*}
\nabla g:R\times X\rightarrow M\text{, }\left( \lambda ,x\right) \mapsto
g\left( \lambda x\right) -\lambda g\left( x\right)
\end{equation*}%
Since $M_{1}$ is closed in the subspace topology on $M$ inherited from $Y$,
we have that there exists $x_{0}\in X_{0}$, $\lambda _{0}\in R$, an open
subring $O$ of $R$, an open subgroup $X_{1}^{\prime }$ in $X$ and $m\in M$
such that%
\begin{equation*}
g\left( \left( \lambda +\lambda _{0}\right) \left( x+x_{0}\right) \right)
\equiv \left( \lambda +\lambda _{0}\right) g\left( x+x_{0}\right) +m\ 
\mathrm{\mathrm{mod}}M_{1}
\end{equation*}%
for every $\lambda \in O$ and $x\in X_{1}^{\prime }$.\ In particular for $%
\lambda =0$ and $x=0$ we obtain%
\begin{equation*}
g\left( \lambda _{0}x_{0}\right) \equiv \lambda _{0}g\left( x_{0}\right) +m%
\text{ }\mathrm{\mathrm{mod}}M_{1}\text{.}
\end{equation*}%
Thus, in particular for $x=0$ and $\lambda \in O$ we obtain%
\begin{eqnarray*}
g\left( \lambda x_{0}\right) +\lambda _{0}g\left( x_{0}\right) +m &\equiv
&g\left( \lambda x_{0}\right) +g\left( \lambda _{0}x_{0}\right) \\
&\equiv &g\left( \left( \lambda +\lambda _{0}\right) x_{0}\right) \\
&\equiv &\left( \lambda +\lambda _{0}\right) g\left( x_{0}\right) +m \\
&\equiv &\lambda g\left( x_{0}\right) +\lambda _{0}g\left( x_{0}\right) +m\ 
\mathrm{mod}M_{1}\text{.}
\end{eqnarray*}%
Thus, we have that 
\begin{equation*}
g\left( \lambda x_{0}\right) \equiv \lambda g\left( x_{0}\right) \ \mathrm{%
mod}M_{1}
\end{equation*}%
for every $\lambda \in O$.

For $\lambda =0$ and $x\in X_{1}^{\prime }$ we obtain%
\begin{eqnarray*}
g\left( \lambda _{0}x\right) +\lambda _{0}g\left( x_{0}\right) +m &\equiv
&g\left( \lambda _{0}x\right) +g\left( \lambda _{0}x_{0}\right) \\
&\equiv &g\left( \lambda _{0}\left( x+x_{0}\right) \right) \\
&\equiv &\lambda _{0}g\left( x+x_{0}\right) +m \\
&\equiv &\lambda _{0}g\left( x\right) +\lambda _{0}g\left( x_{0}\right) +m\ 
\mathrm{mod}M_{1}\text{.}
\end{eqnarray*}%
Thus we have that%
\begin{equation*}
g\left( \lambda _{0}x\right) \equiv \lambda _{0}g\left( x\right) \ \mathrm{%
mod}M_{1}
\end{equation*}%
for $x\in X_{1}^{\prime }$.

Finally for $\lambda \in O$ and $x\in X_{1}^{\prime }$ we have that%
\begin{eqnarray*}
&&g\left( \lambda x\right) +\lambda _{0}g\left( x\right) +\lambda g\left(
x_{0}\right) +\lambda _{0}g\left( x_{0}\right) +m \\
&\equiv &g\left( \lambda x\right) +g\left( \lambda _{0}x\right) +g\left(
\lambda x_{0}\right) +g\left( \lambda _{0}x_{0}\right) \\
&\equiv &g\left( \lambda x+\lambda _{0}x+\lambda x_{0}+\lambda
_{0}x_{0}\right) \\
&\equiv &\lambda g\left( x\right) +\lambda _{0}g\left( x\right) +\lambda
g\left( x_{0}\right) +\lambda _{0}g\left( x_{0}\right) +m\ \mathrm{mod}M_{1}%
\text{.}
\end{eqnarray*}%
Thus, we have that for every $\lambda \in O$ and $x\in X_{1}^{\prime }$,%
\begin{equation*}
g\left( \lambda x\right) \equiv \lambda g\left( x\right) \ \mathrm{\mathrm{%
mod}}M_{1}\text{.}
\end{equation*}

If $S$ is a bounded open subring of $R$ then there exists an open zero
neighborhood $V$ of $R$ such that $V\cdot S\subseteq O$. Thus, for every $%
\lambda \in S$ and $v\in V$ we have that%
\begin{equation*}
g\left( \lambda vx\right) \equiv \lambda vg\left( x\right) \ \mathrm{mod}%
M_{1}
\end{equation*}%
and in particular%
\begin{equation*}
g\left( vx\right) \equiv vg\left( x\right) \ \mathrm{mod}M_{1}
\end{equation*}%
Define now $X_{1}^{\prime \prime }=V\cdot X_{1}^{\prime }$. Then we have
that $X_{1}^{\prime \prime }$ is an open $S$-submodule of $X$ contained in $%
X_{1}$. Furthermore, for $\lambda \in S$ and $x\in X_{1}^{\prime \prime }$
we have that $x=vy$ for some $y\in X_{1}^{\prime }$ and $v\in V$ and hence%
\begin{equation*}
g\left( \lambda x\right) =g\left( \lambda vy\right) \equiv \lambda vg\left(
y\right) \equiv \lambda g\left( vy\right) =\lambda g\left( x\right) \ 
\mathrm{mod}M_{1}\text{.}
\end{equation*}%
This concludes the proof.
\end{proof}

\begin{proposition}
\label{Proposition:approximately-linear-lift}Let $R$ be a non-Archimedean
Polish ring. Suppose that $\varphi :X\rightarrow Y\ $is a Borel-definable $R$%
-homomorphism between $R$-modules with a non-Archimedean Polish cover.
Suppose that $Y=\hat{Y}/M$ where $M$ has a basis of zero neighborhoods that
are closed in the subspace topology inherited from $\hat{Y}$. Then $\varphi $
has an approximately $R$-linear continuous lift.
\end{proposition}

\begin{proof}
This follows immediately from Lemma \ref{Lemma:approximately-linear-lift}
and Proposition \ref{Proposition:approximately-additive}.
\end{proof}

\section{Left hearts of categories of Polish modules\label%
{Section:left-hearts}}

In this section, we provide explicit descriptions of the heart of categories
of Polish modules.

\subsection{The left heart of Polish modules\label{Subsection:LH-A}}

Fix a Polish group or Polish ring $R$. We let $\mathcal{A}_{R}$ be the
category whose objects are the Polish $R$-modules and whose morphisms are
the continuous $R$-homomorphisms.

\begin{lemma}
\label{Lemma:AR-quasiabelian}$\mathcal{A}_{R}$ is a countably complete
quasi-abelian category.
\end{lemma}

\begin{proof}
It is clear that $\mathcal{A}_{R}$ is an additive category. If $\varphi
:X\rightarrow Y$ is a continuous $R$-homomorphism, then its kernel in $%
\mathcal{A}_{R}$ is $\mathrm{\mathrm{ker}}\left( \varphi \right) =\left\{
x\in X:\varphi \left( x\right) =0\right\} $, and its cokernel in $\mathcal{A}%
_{R}$ is the quotient of $Y$ by the \emph{closure }of the image of $\varphi $%
.\ Thus, the kernels in $\mathcal{A}_{R}$ are the continuous injective $R$%
-homomorphism \emph{with closed image}, and the cokernels in $\mathcal{A}%
_{R} $ are the continuous surjective $R$-homomorphisms. It remains to prove
that the class of kernels is stable under push-out along arbitrary
morphisms, and the class of cokernels is stable under pull-back along
arbitrary morphisms.

Suppose that $\varphi _{i}:X\rightarrow Y_{i}$ are continuous $R$%
-homomorphisms for $i\in \left\{ 0,1\right\} $. Let $p_{i}:Y_{i}\rightarrow
Z $ for $i\in \left\{ 0,1\right\} $ be their pushout.\ Thus we have that $Z$
is the quotient of $Y_{0}\oplus Y_{1}$ by the closure $M$ of the $R$%
-submodule $\left\{ \left( -\varphi _{0}\left( x\right) ,\varphi _{1}\left(
x\right) \right) :x\in X\right\} $. The map $p_{0}:Y_{0}\rightarrow Z$ is
defined by $y\mapsto \left( y,0\right) +M$, and the map $p_{1}:Y_{1}%
\rightarrow Z$ is defined by $y\mapsto \left( 0,y\right) +M$. Suppose that $%
\varphi _{0}$ is a kernel, namely it is injective and it has closed range.
We need to prove that $p_{1}$ is also a kernel.

Suppose that $y\in Y_{1}$ is such that $p_{1}\left( y\right) =0$. Thus, we
have that $\left( 0,y\right) \in M$. Hence, there exists a sequence $\left(
x_{n}\right) $ in $X$ such that $\varphi _{0}\left( x_{n}\right) \rightarrow
0$ and $\varphi _{1}\left( x_{n}\right) \rightarrow y$. Since $\varphi _{0}$
is injective with closed range, it is a homeomorphism onto its image.
Therefore, we have that $x_{n}\rightarrow 0$ and hence $y=\mathrm{\mathrm{lim%
}}_{n}{}\varphi _{1}(x_{n})=0$. This shows that $p_{1}$ is injective.

Suppose now that $\left( y_{n}\right) $ is a sequence in $Y$ such that $%
\left( p_{1}\left( y_{n}\right) \right) $ converges in $Z$ to $\left(
z_{0},z_{1}\right) +M$. Thus we can find a sequence $\left( x_{n}\right) $
in $X$ such that $\left( \varphi _{0}\left( x_{n}\right) ,p_{1}\left(
y_{n}\right) -\varphi _{1}\left( x_{n}\right) \right) $ converges in $%
Y_{0}\oplus Y_{1}$ to $\left( z_{0},z_{1}\right) $. Since $\varphi _{0}$ is
a kernel, this implies that $\left( x_{n}\right) $ converges in $X$ to some $%
x\in X$ such that $\varphi _{0}\left( x\right) =z_{0}$.\ Thus, we have that $%
\left( p_{1}\left( y_{n}\right) \right) $ converges to $z_{1}+\varphi
_{1}\left( x\right) $ and hence 
\begin{equation*}
\left( z_{0},z_{1}\right) +M=\left( z_{0}-\varphi _{0}\left( x\right)
,z_{1}+\varphi _{1}\left( x\right) \right) +M=\mathrm{lim}_{n}\left( \left(
0,p_{1}\left( y_{n}\right) \right) +M\right) \text{.}
\end{equation*}%
This shows that $p_{1}$ has closed range, concluding the proof that $p_{1}$
is a kernel.

Suppose now that $\varphi _{i}:Y_{i}\rightarrow Z$ for $i\in \left\{
0,1\right\} $ are continuous $R$-homomorphisms.\ Let $\eta _{i}:X\rightarrow
Y_{i}$ for $i\in \left\{ 0,1\right\} $ be their pullback. Then we have that 
\begin{equation*}
X=\left\{ \left( y_{0},y_{1}\right) \in Y_{0}\oplus Y_{1}:\varphi _{0}\left(
y_{0}\right) =\varphi _{1}\left( y_{1}\right) \right\} \text{.}
\end{equation*}%
The map $\eta _{0}:X\rightarrow Y_{0}$ is given by $\left(
y_{0},y_{1}\right) \mapsto y_{0}$, and the map $\eta _{1}:X\rightarrow Y_{1}$
is given by $\left( y_{0},y_{1}\right) \mapsto y_{1}$. Suppose that $\varphi
_{1}$ is a cokernel, i.e.\ surjective. We need to show that $\eta _{0}$ is
also a cokernel. Suppose that $y_{0}\in Y_{0}$. Consider $\varphi _{0}\left(
y_{0}\right) \in Z$. Since $\varphi _{1}$ is surjective, there exists $%
y_{1}\in Y_{1}$ such that $\varphi _{1}\left( y_{1}\right) =\varphi
_{0}\left( y_{0}\right) $. Thus we have that $x:=\left( y_{0},y_{1}\right)
\in X$ is such that $\eta _{0}\left( x\right) =y_{0}$. This concludes the
proof that $\eta _{0}$ is a cokernel.

Finally, the fact that $\mathcal{A}_{R}$ is countably complete follows from
the fact that it has countable products.
\end{proof}

We let $\mathcal{M}_{R}$ be the category whose objects are the $R$-modules
with a Polish cover, and whose morphisms are the Borel-definable $R$%
-homomorphisms. We regard $\mathcal{A}_{R}$ as a full subcategory of $%
\mathcal{M}_{R}$, by identifying a Polish $R$-module $X$ with the $R$-module
with a Polish cover $X/N$ where $N$ is the trivial submodule of $X$.

\begin{theorem}
\label{Theorem:MR-abelian}The category $\mathcal{M}_{R}$ is abelian. The
inclusion functor $\mathcal{A}_{R}\rightarrow \mathcal{M}_{R}$ is exact and
countably continuous, and it extends to an equivalence of categories $%
\mathrm{LH}\left( \mathcal{A}_{R}\right) \rightarrow \mathcal{M}_{R}$.
\end{theorem}

\begin{proof}
Suppose that $\varphi :X\rightarrow Y$ is a Borel-definable $R$-homomorphism
between $R$-modules with a Polish cover. Then by Proposition \ref%
{Proposition:preimage} and the results of Subsection \ref{Subsection:modules}%
, we have that

\begin{itemize}
\item $\mathrm{ker}\left( \varphi \right) :=\left\{ x\in X:\varphi \left(
x\right) =0\right\} $ is an $R$-submodule with a Polish cover of $X$, and

\item $\varphi \left( X\right) $ is an $R$-submodule with a Polish cover of $%
Y$.
\end{itemize}

It is easy to see that $\mathrm{ker}\left( \varphi \right) \rightarrow X$ is
the kernel of $\varphi $ and $Y\rightarrow Y/\varphi \left( X\right) $ is
the cokernel of $\varphi $. This easily implies that every monic arrow is a
kernel and every epic arrow is a cokernel, and hence $\mathcal{M}_{R}$ is an
abelian category.

It follows from the characterization of the left heart provided by the last
item in Proposition \ref{Proposition:left-heart} that the inclusion $J:%
\mathcal{A}_{R}\rightarrow \mathcal{M}_{R}$ extends to an equivalence of
categories $\mathrm{LH}\left( \mathcal{A}_{R}\right) \rightarrow \mathcal{M}%
_{R}$. It is also easy to see that $\mathcal{M}_{R}$ has countable products,
and that $J$ preserves countable products. Since $J$ is finitely continuous
and preserves countable products, it is also countably continuous.
\end{proof}

\subsection{Left hearts of categories of Polish modules\label%
{Subsection:LH-B}}

Recall that a subcategory $\mathcal{C}$ of a category $\mathcal{D}$ is \emph{%
strictly full} if its collection of objects is closed under isomorphism in $%
\mathcal{D}$, and for objects $x,y$ in $\mathcal{C}$, $\mathrm{Hom}_{%
\mathcal{C}}\left( x,y\right) =\mathrm{Hom}_{\mathcal{D}}\left( x,y\right) $%
. Let $\mathcal{B}$ be a\emph{\ }strictly full quasi-abelian subcategory of
the category $\mathcal{A}_{R}$ of Polish $R$-modules. We notice that this
implies that, if $X,Y$ are objects of $\mathcal{B}$, and $f:X\rightarrow Y$
is a continuous $R$-homomorphism, then $f$ is a morphism in $\mathcal{B}$, $%
\mathrm{k\mathrm{er}}\left( f\right) $ is an object of $\mathcal{B}$, and $%
f\left( X\right) $ endowed with the Polish $R$-module topology induced via $%
f $ by the Polish $R$-module topology on $X$ is an object of $\mathcal{B}$
(being isomorphic to \textrm{coker}$\left( \mathrm{ker}\left( f\right)
\right) $).

\begin{definition}
An $R$-module with a $\mathcal{B}$-cover is an $R$-module $X$ explicitly
presented as a quotient $\hat{X}/N$ where $\hat{X}$ and $N$ are objects of $%
\mathcal{B}$, $N$ is an $R$-submodule of $\hat{X}$, and the inclusion $%
N\rightarrow \hat{X}$ is continuous.

If $X=\hat{X}/N$ is an $R$-module with a $\mathcal{B}$-cover, then an $R$%
-submodule with a $\mathcal{B}$-cover of $X$ is an $R$-submodule $Y=\hat{Y}%
/N $ where $\hat{Y}$ is an $R$-submodule of $\hat{X}$, $\hat{Y}$ is an
object of $\mathcal{B}$, and the inclusion $\hat{Y}\rightarrow \hat{X}$ is
continuous.
\end{definition}

\begin{definition}
\label{Definition:B-definable}An $R$-homomorphism $\varphi :X\rightarrow Y$
between $R$-modules with a $\mathcal{B}$-cover is:

\begin{itemize}
\item $\mathcal{B}$-\emph{definable} if its graph $\Gamma \left( \varphi
\right) $ is an $R$-submodule with a $\mathcal{B}$-cover of $X\oplus Y$;

\item \emph{liftable} if has a lift to a continuous $R$-homomorphism $\hat{X}%
\rightarrow \hat{Y}$, where $X=\hat{X}/N$ and $Y=\hat{Y}/M$.
\end{itemize}
\end{definition}

\begin{remark}
\label{Remark:inverse-B-definable}If $\varphi :X\rightarrow Y$ is a $%
\mathcal{B}$-definable bijective $R$-homomorphism, then $\varphi ^{-1}$ is $%
\mathcal{B}$-definable.
\end{remark}

\begin{remark}
By Theorem \ref{Theorem:factorization}, an $R$-homomorphism is $\mathcal{A}%
_{R}$-definable if and only if it is Borel-definable.
\end{remark}

\begin{lemma}
A liftable $R$-homomorphism is $\mathcal{B}$-definable.
\end{lemma}

\begin{proof}
Suppose that $\varphi :\hat{X}/N\rightarrow \hat{Y}/M$, where $X=\hat{X}/N$
and $Y=\hat{Y}/M$, is induced by a continuous $R$-homomorphism $f:\hat{X}%
\rightarrow \hat{Y}$.\ Then we have that%
\begin{equation*}
W:=\{\left( x,y,z\right) \in \hat{X}\oplus \hat{Y}\oplus M:f\left( x\right)
=y+z\}
\end{equation*}%
is an object of $\mathcal{B}$ being the kernel of the morphism%
\begin{equation*}
\hat{X}\oplus \hat{Y}\oplus M\rightarrow \hat{Y}\text{, }\left( x,y,z\right)
\mapsto f\left( x\right) -y-z
\end{equation*}%
in $\mathcal{B}$. Thus, we have that%
\begin{equation*}
\{\left( x,y\right) \in \hat{X}\oplus \hat{Y}:f\left( x\right) \equiv y\ 
\mathrm{mod}M\}
\end{equation*}%
is a Polishable $R$-submodule of $\hat{X}\oplus \hat{Y}$ that belongs to $%
\mathcal{B}$, being the image of $W$ under the continuous $R$-homomorphism%
\begin{equation*}
W\rightarrow \hat{X}\oplus \hat{Y}\text{, }\left( x,y,z\right) \mapsto
\left( x,y\right) \text{.}
\end{equation*}
\end{proof}

\begin{lemma}
Suppose that $\varphi :X\rightarrow Y$ and $\psi :Y\rightarrow Z$ are $%
\mathcal{B}$-definable $R$-homomorphisms. Then $\psi \circ \varphi
:X\rightarrow Z$ is $\mathcal{B}$-definable.
\end{lemma}

\begin{proof}
Suppose that $\varphi ,\psi $ are $\mathcal{B}$-definable. By assumption, we
have that $\Gamma \left( \varphi \right) \subseteq X\oplus Y$ and $\Gamma
\left( \psi \right) \subseteq Y\oplus Z$ are $R$-submodules with a Polish
cover. Then we have that 
\begin{equation*}
W:=\left\{ \left( x,y_{0},y_{1},z\right) :\left( x,y_{0}\right) \in \Gamma
\left( \varphi \right) ,\left( y_{1},z\right) \in \Gamma \left( \psi \right)
,y_{0}=y_{1}\right\}
\end{equation*}%
is an $R$-submodule with a $\mathcal{B}$-cover of $\Gamma \left( \varphi
\right) \oplus \Gamma \left( \psi \right) $, being the kernel of the
liftable $R$-homomorphism%
\begin{equation*}
\Gamma \left( \varphi \right) \oplus \Gamma \left( \psi \right) \rightarrow Y%
\text{, }\left( x,y_{0},y_{1},z\right) \mapsto y_{0}-y_{1}\text{.}
\end{equation*}%
It follows that $\Gamma \left( \psi \circ \varphi \right) $ is an $R$%
-submodule with a $\mathcal{B}$-cover of $X\oplus Z$, being the image of $W$
under the liftable $R$-homomorphism%
\begin{equation*}
W\rightarrow X\oplus Z\text{, }\left( x,y,y,z\right) \mapsto \left(
x,z\right) \text{.}
\end{equation*}
\end{proof}

\begin{lemma}
Suppose that $\varphi :X\rightarrow Y$ is a $\mathcal{B}$-definable $R$%
-homomorphism. Then $\mathrm{ker}\left( \varphi \right) \subseteq X$ and $%
\varphi \left( X\right) \subseteq Y$ are $R$-submodules with a $\mathcal{B}$%
-cover.
\end{lemma}

\begin{proof}
We have that%
\begin{equation*}
W:=\left\{ \left( x,y\right) \in \Gamma \left( \varphi \right) :y=0\right\}
\end{equation*}%
is an $R$-submodule with a $\mathcal{B}$-cover of $\Gamma \left( \varphi
\right) $, and hence \textrm{Ker}$\left( \varphi \right) \subseteq X$ is an $%
R$-submodule with a $\mathcal{B}$-cover, being the image of the liftable $R$%
-module homomorphism $W\rightarrow X$, $\left( x,y\right) \mapsto x$.

Similarly, we have that $\varphi \left( X\right) $ is the image of the
liftable $R$-module homomorphism $\Gamma \left( \varphi \right) \rightarrow
Y $, $\left( x,y\right) \mapsto y$.
\end{proof}

\begin{lemma}
\label{Lemma:characterize-B-definable}If $\varphi :X\rightarrow Y$ is an $R$%
-homomorphism between $R$-modules with a $\mathcal{B}$-cover, the following
are equivalent:

\begin{enumerate}
\item $\varphi $ is $\mathcal{B}$-definable;

\item there exist an $R$-module with a $\mathcal{B}$-cover $Z$ and liftable $%
R$-homomorphisms $\psi :Z\rightarrow Y$ and $\sigma :Z\rightarrow X$ such
that $\sigma $ is a bijection and $\varphi =\varphi ^{\prime }\circ \sigma
^{-1}$.
\end{enumerate}
\end{lemma}

\begin{proof}
Suppose that $\varphi :X\rightarrow Y$ is an $R$-module homomorphism between 
$R$-modules with a $\mathcal{B}$-cover. Suppose that $\varphi $ is $\mathcal{%
B}$-definable. Then we can set $Z=\Gamma \left( \varphi \right) $, $\sigma
:Z\rightarrow X$, $\left( x,y\right) \mapsto x$, and $\psi :Z\rightarrow Y$, 
$\left( x,y\right) \mapsto y$. This shows that (1) implies (2). Suppose now
that $\varphi $ satisfies (2). Then we have that $\psi $ and $\sigma $ are $%
\mathcal{B}$-definable, being liftable. Hence $\varphi =\psi \circ \sigma
^{-1}$ is $\mathcal{B}$-definable, and (1) holds.
\end{proof}

\begin{corollary}
\label{Corollary:characterize-B-definable}If $\varphi :X\rightarrow Y$ is a $%
\mathcal{B}$-definable $R$-homomorphism between $R$-modules with a $\mathcal{%
B}$-cover, then it is Borel-definable.
\end{corollary}

\begin{proof}
By\ Lemma \ref{Lemma:characterize-B-definable}, adopting the notation as in\
(2) from its statement, we have that $\varphi =\psi \circ \sigma ^{-1}$.
Since $\psi $ and $\sigma $ are liftable, we have that $\psi $ and (by
Remark \ref{Remark:inverse}) $\sigma ^{-1}$ are Borel-definable.\ Hence, $%
\varphi $ is Borel-definable.
\end{proof}

\begin{lemma}
\label{Lemma:sum-definable}If $\varphi _{0},\varphi _{1}:X\rightarrow Y$ are 
$\mathcal{B}$-definable $R$-homomorphisms, then $\varphi _{0}+\varphi _{1}$
is $\mathcal{B}$-definable.
\end{lemma}

\begin{proof}
Suppose that $\varphi _{0},\varphi _{1}:X\rightarrow Y$ are $\mathcal{B}$%
-definable $R$-homomorphism. Then there exist objects $X^{\prime },X^{\prime
\prime }$ of $\mathcal{B}$ and continuous $R$-homomorphisms $\varphi
_{0}^{\prime }:X^{\prime }\rightarrow Y$, $\varphi _{1}^{\prime }:X^{\prime
\prime }\rightarrow Y$, $\sigma :X^{\prime }\rightarrow X$, and $\tau
:X^{\prime \prime }\rightarrow X$ such that $\sigma ,\tau $ are bijective, $%
\varphi _{0}=\varphi _{0}^{\prime }\circ \sigma ^{-1}$, and $\varphi
_{1}=\varphi _{1}^{\prime }\circ \tau ^{-1}$. We can thus consider the
continuous $R$-homomorphism $\psi :X^{\prime }\oplus X^{\prime \prime
}\rightarrow Y$, $\left( x,y\right) \mapsto \varphi _{0}^{\prime }\left(
x\right) +\varphi _{1}^{\prime }\left( y\right) $, and the continuous
bijective $R$-homomorphism $\lambda :X^{\prime }\oplus X^{\prime \prime
}\rightarrow X\oplus X$, $\left( x,y\right) \mapsto \left( \sigma \left(
x\right) ,\tau \left( y\right) \right) $.\ Then we have that $\varphi
_{0}+\varphi _{1}=\psi \circ \lambda ^{-1}\circ \Delta $ where $\Delta
:X\rightarrow X\oplus X$, $x\mapsto \left( x,x\right) $. This shows that $%
\varphi _{0}+\varphi _{1}$ is $\mathcal{B}$-definable.
\end{proof}

Define $\mathcal{M}_{\mathcal{B}}$ to be the (not necessarily full)
subcategory of $\mathcal{M}_{R}$ whose objects are the $R$-modules with a $%
\mathcal{B}$-cover and whose morphisms are the $\mathcal{B}$-definable $R$%
-homomorphisms.

\begin{theorem}
\label{Theorem:left-heart-B}Let $\mathcal{B}$ be a strictly full
quasi-abelian subcategory of the category $\mathcal{A}_{R}$ of Polish $R$%
-modules and continuous $R$-module homomorphisms. Let $\mathcal{M}_{\mathcal{%
B}}$ be the category of $R$-modules with a $\mathcal{B}$-cover and $\mathcal{%
B}$-definable $R$-homomorphisms. Then we have that:

\begin{enumerate}
\item $\mathcal{M}_{\mathcal{B}}$ is an abelian category;

\item the inclusion $\mathcal{B}\rightarrow \mathcal{M}_{\mathcal{B}}$
extends to an equivalence of category $\mathrm{LH}\left( \mathcal{B}\right)
\rightarrow \mathcal{M}_{\mathcal{B}}$;

\item if $\mathcal{B}$ is countably complete and the inclusion $\mathcal{B}%
\rightarrow \mathcal{A}_{R}$ is countably continuous, then $\mathcal{M}_{%
\mathcal{B}}$ is countably complete and the inclusions $\mathcal{B}%
\rightarrow \mathcal{M}_{\mathcal{B}}\rightarrow \mathcal{M}_{R}$ are
countably continuous.
\end{enumerate}
\end{theorem}

\begin{proof}
(1) We begin with showing that $\mathcal{M}_{\mathcal{B}}$ is an additive
subcategory of $\mathcal{M}_{R}$. It is clear that the zero object for $%
\mathcal{M}_{R}$ is also the zero object for $\mathcal{M}_{\mathcal{B}}$. By
Lemma \ref{Lemma:sum-definable} and Corollary \ref%
{Corollary:characterize-B-definable}, the set of $\mathcal{B}$-definable $R$%
-homomorphisms $X\rightarrow Y$ is a subgroup of the set of Borel-definable $%
R$-homomorphisms. It remains to prove that, for objects $X,Y$ in $\mathcal{M}%
_{\mathcal{B}}$, their biproduct $X\oplus Y$ in $\mathcal{M}_{R}$ is also
their coproduct in $\mathcal{M}_{\mathcal{B}}$. Since $\mathcal{B}$ is a
quasi-abelian subcategory of $\mathcal{A}_{R}$, we have that $X\oplus Y$ is
an $R$-module with a $\mathcal{B}$-cover. Since every liftable $R$%
-homomorphism is $\mathcal{B}$-definable, we have that the canonical maps $%
X\rightarrow X\oplus Y$ and $Y\rightarrow X\oplus Y$ are $\mathcal{B}$%
-definable. It remains to prove that $X\oplus Y$ satisfies the universal
property of the coproduct. Suppose that $\varphi :Z\rightarrow X$ and $\psi
:Z\rightarrow Y$ are $\mathcal{B}$-definable $R$-homomorphisms. Let $\varphi
\oplus \psi :Z\rightarrow X\oplus Y$ be the corresponding Borel-definable $R$%
-homomorphism. We need to prove that $\varphi \oplus \psi $ is $\mathcal{B}$%
-definable. We have that $\varphi =\varphi ^{\prime }\circ \sigma ^{-1}$ and 
$\psi =\psi ^{\prime }\circ \tau ^{-1}$ for some $R$-modules with a $%
\mathcal{B}$-cover $Z^{\prime },Z^{\prime \prime }$ and liftable
homomorphisms $\sigma :Z^{\prime }\rightarrow Z$, $\varphi ^{\prime
}:Z^{\prime }\rightarrow X$, $\tau :Z^{\prime \prime }\rightarrow Z$, $\psi
^{\prime }:Z^{\prime \prime }\rightarrow Y$ such that $\sigma ,\tau $ are
bijective. Thus we have that 
\begin{equation*}
\varphi \oplus \psi =\left( \varphi ^{\prime }\oplus \psi ^{\prime }\right)
\circ \left( \sigma \oplus \tau \right) ^{-1}
\end{equation*}%
is $\mathcal{B}$-definable since $\varphi ^{\prime }\oplus \psi ^{\prime }$
and $\sigma \oplus \tau $ are liftable and $\sigma \oplus \tau $ is
bijective. This shows that $\mathcal{M}_{\mathcal{B}}$ is an additive
subcategory of $\mathcal{M}_{R}$.

We now prove that $\mathcal{M}_{\mathcal{B}}$ is an abelian category, which
is furthermore an abelian subcategory of $\mathcal{M}_{R}$. Suppose that $%
\varphi :X\rightarrow Y$ is a $\mathcal{B}$-definable $R$-homomorphism.We
have that $\mathrm{ker}\left( \varphi \right) $ is an $R$-module with a $%
\mathcal{B}$-cover. Let $\iota :\mathrm{ker}\left( \varphi \right)
\rightarrow X$ be the inclusion map. We now show that $\mathrm{\mathrm{%
\mathrm{ke}}r}\left( \varphi \right) $ is the kernel of $\varphi $ in $%
\mathcal{M}_{\mathcal{B}}$. Suppose that $\psi :Z\rightarrow X$ is a $%
\mathcal{B}$-definable $R$-homomorphism such that $\varphi \circ \psi =0$.\
We can write $\psi =\psi ^{\prime }\circ \sigma ^{-1}$ where $\sigma
:Z^{\prime }\rightarrow Z$ and $\psi ^{\prime }:Z^{\prime }\rightarrow X$
are liftable $R$-homomorphism such that $\sigma $ is bijective and $%
Z^{\prime }$ is an $R$-module with a $\mathcal{B}$-cover. Since $\varphi
\circ \psi =0$ we have that $\varphi \left( Z\right) \subseteq \mathrm{ker}%
\left( \varphi \right) $ and hence $\psi ^{\prime }\left( W\right) \subseteq 
\mathrm{ker}\left( \varphi \right) $. Thus, we can write $\psi ^{\prime
}=\iota \circ \psi ^{\prime \prime }$ for a liftable $R$-homomorphism $\psi
^{\prime \prime }:W\rightarrow \mathrm{ker}\left( \varphi \right) $. Hence,
we have that $\psi ^{\prime \prime }\circ \sigma ^{-1}:Z\rightarrow \mathrm{%
ker}\left( \varphi \right) $ is a $\mathcal{B}$-definable $R$-homomorphism
such that $\iota \circ \left( \psi ^{\prime \prime }\circ \sigma
^{-1}\right) =\psi $. This concludes the proof that $\mathrm{ker}\left(
\varphi \right) $ is the kernel of $\varphi $ in $\mathcal{M}_{\mathcal{B}}$.

Suppose that $\varphi :X\rightarrow Y$ is a $\mathcal{B}$-definable $R$%
-homomorphism. We show that $Y/\varphi \left( X\right) $ is the cokernel of $%
\varphi $ in $\mathcal{M}_{\mathcal{B}}$. We can write $\varphi =\varphi
^{\prime }\circ \sigma ^{-1}$ for some liftable $R$-homomorphisms $\varphi
^{\prime }:X^{\prime }\rightarrow Y$ and $\sigma :X^{\prime }\rightarrow X$
where $\sigma $ is bijective and $X^{\prime }$ is an $R$-module with a $%
\mathcal{B}$-cover. Thus, we have that $\varphi \left( X\right) =\varphi
^{\prime }\left( X^{\prime }\right) \subseteq Y$ is a $R$-submodule with a $%
\mathcal{B}$-cover since $\varphi ^{\prime }$ is liftable and $\mathcal{B}$
is closed under taking quotients by closed $R$-submodules. Hence, $Y/\varphi
\left( X\right) $ is an $R$-module with a $\mathcal{B}$-cover. Suppose now
that $\psi :Y\rightarrow Z$ is a $\mathcal{B}$-definable $R$-homomorphism
such that $\psi \circ \varphi =0$. We can write $\psi =\psi ^{\prime }\circ
\tau ^{-1}$, where $\psi ^{\prime }:Y^{\prime }\rightarrow Z$ and $\tau
:Y^{\prime }\rightarrow Y$ are liftable $R$-homomorphisms, $\tau $ is
bijective, and $Y^{\prime }$ is an $R$-module with a $\mathcal{B}$-cover.
Then we have that $\left( \tau ^{-1}\circ \varphi \right) \left( X\right)
\subseteq Y^{\prime }$ is an $R$-submodule with a $\mathcal{B}$-cover of $%
Y^{\prime }$. Furthermore $0=\psi \circ \varphi =\psi ^{\prime }\circ \left(
\tau ^{-1}\circ \varphi \right) $ and hence $\left( \tau ^{-1}\circ \varphi
\right) \left( X\right) \subseteq \mathrm{ker}\left( \psi ^{\prime }\right) $%
. Since $\psi ^{\prime }$ is liftable, it induces a liftable $R$%
-homomorphism $\bar{\psi}^{\prime }:Y^{\prime }/(\tau ^{-1}\circ \varphi
)\left( X\right) \rightarrow Z$. Similarly, we have that $\tau \left( \left(
\tau ^{-1}\circ \sigma \right) \left( X\right) \right) =\sigma \left(
X\right) $ since $\tau :Y^{\prime }\rightarrow Y$ is a bijective $R$-linear
homomorphism. Hence, being liftable, it induces a liftable bijective $R$%
-linear homomorphism $\bar{\tau}:Y^{\prime }/\left( \tau ^{-1}\circ \varphi
\right) \left( X\right) \rightarrow Y/\varphi \left( X\right) $. Thus, we
have that $\bar{\psi}^{\prime }\circ \bar{\tau}^{-1}:Y/\varphi \left(
X\right) \rightarrow Z$ is a $\mathcal{B}$-definable $R$-homomorphism such
that, letting $\pi _{Y}:Y\rightarrow Y/\varphi \left( X\right) $ and $\pi
_{Y^{\prime }}:Y^{\prime }\rightarrow Y^{\prime }/\left( \tau ^{-1}\circ
\varphi \right) \left( X\right) $ be the canonical quotient mappings, 
\begin{equation*}
\bar{\psi}^{\prime }\circ \bar{\tau}^{-1}\circ \pi _{Y}=\bar{\psi}^{\prime
}\circ \pi _{Y^{\prime }}\circ \tau ^{-1}=\psi ^{\prime }\circ \tau
^{-1}=\psi \text{.}
\end{equation*}
This concludes the proof that $Y/\varphi \left( X\right) $ is the cokernel
of $\varphi $ in $\mathcal{M}_{\mathcal{B}}$.

(2) This follows immediately from (1) and the characterization of the left
heart of a quasi-abelian category from the last item in Proposition \ref%
{Proposition:left-heart}.

(3) Suppose that $\left( X_{n}\right) _{n\in \omega }$ is a sequence of $R$%
-modules with a $\mathcal{B}$-cover. It suffices to prove that their product 
$X_{\omega }:=\prod_{n\in \omega }X_{n}$ in $\mathcal{M}_{R}$ is also their
product in $\mathcal{M}_{\mathcal{B}}$. We have that $X_{\omega }$ is an $R$%
-module with a $\mathcal{B}$-cover, since $\mathcal{B}$ is countably
complete and the inclusion $\mathcal{B}\rightarrow \mathcal{A}_{R}$ is
countably continuous. Furthermore the projection maps $\pi _{i}:X_{\omega
}\rightarrow X_{i}$ for $i\in \omega $ are liftable, and hence $\mathcal{B}$%
-definable. Suppose now that $Z$ is an $R$-module with a $\mathcal{B}$%
-cover, and $\varphi _{i}:X_{i}\rightarrow Z$ are $\mathcal{B}$-definable $R$%
-homomorphisms for $i\in \omega $. Then we have that there exist $R$-modules
with a $\mathcal{B}$-cover $X_{i}^{\prime }$, liftable $R$-homomorphisms $%
\sigma _{i}:X_{i}^{\prime }\rightarrow X_{i}$ and $\varphi _{i}^{\prime
}:X_{i}^{\prime }\rightarrow Z$ such that $\sigma _{i}$ is bijective and $%
\varphi _{i}=\varphi _{i}^{\prime }\circ \sigma _{i}^{-1}$. Define then $%
X_{\omega }^{\prime }:=\prod_{i\in \omega }X_{\omega }^{\prime }$. Then we
have that the sequences $\left( \sigma _{i}\right) $ and $\left( \varphi
_{i}^{\prime }\right) $ induce liftable $R$-homomorphisms $\sigma :X_{\omega
}^{\prime }\rightarrow X_{\omega }$ and $\varphi ^{\prime }:X_{\omega
}^{\prime }\rightarrow X$. Setting $\varphi :=\varphi ^{\prime }\circ \sigma
^{-1}:X_{\omega }\rightarrow Y$ we obtain a $\mathcal{B}$-definable $R$%
-homomorphism such that $\varphi \circ \pi _{i}=\varphi _{i}$ for every $%
i\in \omega $. This shows that $X_{\omega }$ is the product of $\left(
X_{n}\right) $ in $\mathcal{M}_{\mathcal{B}}$, concluding the proof.
\end{proof}

\subsection{Examples\label{Subsection:examples}}

In this section we apply Theorem \ref{Theorem:left-heart-B} to describe the
left heart of a number of important categories of Polish $R$-modules as a
full subcategory of the category of $R$-modules with a Polish cover.

\begin{definition}
\label{Definition:extensions}Suppose that $\mathcal{B}$ is a strictly full
quasi-abelian subcategory of $\mathcal{A}_{R}$. We say that $\mathcal{B}$ is
a \emph{thick subcategory }of $\mathcal{A}_{R}$ \cite[Definition 8.3.21(iv)]%
{kashiwara_categories_2006} if it is closed under extensions, i.e., for
every short exact sequence 
\begin{equation*}
0\rightarrow A\rightarrow B\rightarrow C\rightarrow 0
\end{equation*}%
of Polish $R$-modules, we have that if $A$ and $C$ are in $\mathcal{B}$,
then $B$ is in $\mathcal{B}$ as well.
\end{definition}

\begin{proposition}
\label{Proposition:extensions}Suppose that $\mathcal{B}$ is a thick
subcategory of $\mathcal{A}_{R}$. An $R$-homomorphism between $R$-modules
with a $\mathcal{B}$-cover is $\mathcal{B}$-definable if and only if it is
Borel-definable.
\end{proposition}

\begin{proof}
Suppose that $\varphi :\hat{X}/N\rightarrow \hat{Y}/M$ is Borel-definable,
where $\hat{X}/N$ and $\hat{Y}/M$ are $R$-modules with a $\mathcal{B}$%
-cover. Then we have a short exact sequence%
\begin{equation*}
0\rightarrow \left\{ 0\right\} \oplus M\rightarrow \hat{\Gamma}\left(
\varphi \right) \rightarrow \hat{X}\rightarrow 0
\end{equation*}%
where $\hat{\Gamma}\left( \varphi \right) $ is the lift to $\hat{X}\oplus 
\hat{Y}$ of the graph of $\varphi $. We have that $\hat{\Gamma}\left(
\varphi \right) $ is Polish by Theorem \ref{Theorem:factorization}. Since $%
\mathcal{B}$ is a thick subcategory of $\mathcal{A}_{R}$, this implies that $%
\hat{\Gamma}\left( \varphi \right) $ is in $\mathcal{B}$, and hence $\varphi 
$ is $\mathcal{B}$-definable.
\end{proof}

The same argument as in the previous proposition shows the following.

\begin{proposition}
Suppose that $\mathcal{B}$ is a thick subcategory of $\mathcal{A}_{R}$. Let $%
\varphi :\hat{X}/N\rightarrow \hat{Y}/M$ be a Borel-definable $R$%
-homomorphism between $R$-modules with a Polish cover, where $\hat{X}$, $N$,
and $M$ are in $\mathcal{B}$. Then $\hat{Y}$ is in $\mathcal{B}$.
\end{proposition}

\begin{proof}
Adopting the notation of Proposition \ref{Proposition:extensions}, we have
that $\hat{\Gamma}\left( \varphi \right) $ is in $\mathcal{B}$, being an
extension of objects of $\mathcal{B}$. Considering the short exact sequence%
\begin{equation*}
0\rightarrow N\oplus M\rightarrow \hat{\Gamma}\left( \varphi \right)
\rightarrow Y\rightarrow 0
\end{equation*}%
shows that $Y$ is in $\mathcal{B}$ as well.
\end{proof}

As an application of Proposition \ref{Proposition:extensions}, we obtain as
a particular instance of Theorem \ref{Theorem:left-heart-B} a description of
the left heart of a number of categories of Polish modules. We refer to \cite%
{perez-garcia_locally_2010} for the theory of Fr\'{e}chet and Banach spaces
over a non-Archimedean valued field. Recall that a Polish abelian group $G$:

\begin{itemize}
\item is \emph{compactly generated }if it has a compact generating set;

\item is a \emph{topological torsion group} if for every $x\in G$, $\mathrm{%
\mathrm{lim}}_{n\rightarrow \infty }n!x=0$---see \cite[Chapter 3]%
{armacost_structure_1981};

\item a\emph{\ topological }$p$\emph{-group}, for some prime $p$, if for
every $x\in G$, $\mathrm{\mathrm{lim}}_{n\rightarrow \infty }p^{n}x=0$---see 
\cite[Chapter 2]{armacost_structure_1981}.
\end{itemize}

Recall also the notion of locally compact abelian Polish groups that \emph{%
has finite ranks }according to \cite[Definition 2.6]%
{hoffmann_homological_2007}. We let the \emph{dimension }of a locally
compact Polish space to be its \emph{covering dimension}.

\begin{theorem}
\label{Theorem:left-heart-B2}Let $R$ be a Polish ring. Let $\mathcal{B}$ one
of the following full subcategories of $\mathcal{A}_{R}$:

\begin{enumerate}
\item non-Archimedean Polish $R$-modules;

\item locally compact Polish $R$-modules;

\item finite-dimensional locally compact Polish $R$-modules;

\item if $R$ is a field, locally bounded Polish vector spaces over $R$;

\item if $R$ is a separable non-Archimedean valued field, separable Fr\'{e}%
chet spaces over $R$;

\item if $R$ is a separable non-Archimedean valued field, separable Banach
spaces over $R$.
\end{enumerate}

For $R=\mathbb{Z}$, let $\mathcal{B}$ one of the following subcategories of
the category $\mathcal{A}_{\mathbb{Z}}$ of Polish abelian groups:

\begin{enumerate}
\item \setcounter{enumi}{6}abelian real Lie groups;

\item totally disconnected locally compact Polish abelian groups;

\item locally compact Polish abelian groups with no small subgroups;

\item compactly generated Polish abelian groups;

\item locally compact Polish topological torsion abelian groups;

\item locally compact Polish topological abelian $p$-groups, for a given
prime number $p$;

\item locally compact Polish abelian groups that have finite ranks.
\end{enumerate}

Then $\mathcal{B}$ is a thick subcategory of $\mathcal{A}_{R}$, and the
inclusion $\mathcal{B}\rightarrow \mathcal{A}_{R}$ extends to a fully
faithful functor $\mathrm{LH}\left( \mathcal{B}\right) \rightarrow \mathrm{LH%
}\left( \mathcal{A}_{R}\right) =\mathcal{M}_{R}$. Thus, the left heart of $%
\mathcal{B}$ is equivalent to the category that has $R$-modules with a $%
\mathcal{B}$-cover as objects and Borel-definable $R$-homomorphisms as
morphisms.
\end{theorem}

\begin{proof}
We show that each of these categories is a thick subcategory of $\mathcal{A}%
_{R}$, and apply Proposition \ref{Proposition:extensions} and Theorem \ref%
{Theorem:left-heart-B}.

\begin{enumerate}
\item Let 
\begin{equation*}
0\rightarrow A\rightarrow X\rightarrow C\rightarrow 0
\end{equation*}%
be an extension of Polish $R$-modules, where $A,C$ are non-Archimedean. We
identify $A$ with a closed submodule of $X$. By \cite[Proposition 4.6]%
{bergfalk_definable_2020}, there exists a continuous function $\phi
:C\rightarrow X$ that is right inverse for the quotient map $\pi
:X\rightarrow C$; see Remark \ref{Remark:automatic-continuity}. Define $%
\kappa \left( x,y\right) =\delta \phi \left( x,y\right) =\phi \left(
x+y\right) -\phi \left( x\right) -\phi \left( y\right) $. By Proposition \ref%
{Proposition:approximately-linear-lift} one can choose $\phi $ such that for
every open zero neighborhood $U$ of $A$ there exist an open subring $O$ of $%
R $ and an open $O$-submodule $V$ of $C$ such that $\kappa \left( x,y\right)
\in U$ and $\phi \left( \lambda x\right) -\lambda \phi \left( x\right) \in U$
for $x,y\in V$ and $\lambda \in O$.

Consider the abelian Polish group $A\oplus _{\kappa }C$ that is equal to $%
A\times C$ as a Polish space, with group operation defined by $\left(
a,c\right) +\left( a^{\prime },c^{\prime }\right) =\left( a+a^{\prime
}+\kappa \left( c,c^{\prime }\right) ,c+c^{\prime }\right) $. Then the
function $X\rightarrow A\oplus _{\kappa }C$, $x\mapsto \left( x-\phi \pi
\left( x\right) ,\pi \left( x\right) \right) $ is a Borel group isomorphism
with inverse $A\oplus _{\kappa }C\rightarrow X$, $\left( a,c\right) \mapsto
a+\phi \left( c\right) $, and hence it is a homeomorphism. This shows that $%
\phi \left( C\right) $ is a closed subset of $X$, and the sets of the form $%
U_{A}+\phi \left( U_{C}\right) $, where $U_{A}$ is a zero neighborhood in $A$
and $U_{C}$ is a zero neighborhood in $C$, form a basis of zero
neighborhoods for $X$. Let $U_{X}$ be a zero neighborhood in $X$. Consider
an open subring $O$ of $R$ and open $O$-submodules $U_{A}$ of $A$ and $U_{C}$
of $C$ such that $U_{A}+\phi \left( U_{C}\right) \subseteq U_{X}$ and for
every $x,y\in V$ and $\lambda \in P$, $\kappa \left( x,y\right) \in U_{A}$
and $\phi \left( \lambda x\right) -\lambda \phi \left( x\right) \in U_{A}$.
Then we have that $U_{A}+\phi \left( U_{C}\right) $ is an open $O$-submodule
of $X$ contained in $U$. This proves that $X$ is non-Archimedean.

\item It is the content of \cite[Theorem 5.22, Theorem 5.25]%
{hewitt_abstract_1979} that the category of locally compact Polish groups is
thick.

\item If $X$ is a Polish $R$-module and $Y$ is a closed $R$-submodule, then
we have that%
\begin{equation*}
\mathrm{\dim }\left( X\right) =\mathrm{\dim }\left( Y\right) +\mathrm{\dim }%
\left( X/Y\right) \text{;}
\end{equation*}%
see \cite{nagami_dimension-theoretical_1962}. Thus, we have that $X$ is
finite-dimensional if and only if $Y$ and $X/Y$ are finite-dimensional.

\item Let $R$ be a field, and 
\begin{equation*}
0\rightarrow M\rightarrow X\rightarrow Y\rightarrow 0
\end{equation*}%
be an extension of Polish $R$-vector spaces, where $M$ and $Y$ are locally
bounded. Since the quotient map $\pi :X\rightarrow Y$ is open, $M$ is closed
in $X$, and $M,Y$ are locally bounded, we have that there exists an open
zero neighborhood $A$ in $X$ such that $A\cap M$ and $\pi \left( A\right) $
are bounded in $M$ and $Y$, respectively. Fix a decreasing neighborhood
basis $\left( W_{i}\right) $ for $0$ in $R$. Fix $i_{0}\in \mathbb{N}$ and
an open zero neighborhood $B\subseteq A$ in $X$ such that $%
B+W_{i_{0}}B\subseteq A$. We claim that $B$ is bounded in $X$. Suppose that
this is not the case. Then there exist a zero neighborhood $V_{0}$ in $X$, a
vanishing sequence $\left( \lambda _{i}\right) $ in $R$, and a sequence $%
\left( b_{i}\right) $ in $X$ such that $\lambda _{i}\in W_{i}$, $b_{i}\in B$%
, and $\lambda _{i}b_{i}\notin V_{0}$ for $i\in \mathbb{N}$.

Fix a zero neighborhood $V_{0}$ in $X$. Fix $i_{1}\geq i_{0}$ and a zero
neighborhood $V_{1}$ in $X$ such that $V_{1}+\left( W_{i_{1}}V_{1}\cup
V_{1}\right) \subseteq V_{0}$. Fix $i\geq i_{1}$. Since $\pi \left( A\right) 
$ is bounded in $Y$, there exists $\alpha _{i}\geq i$ such that $\lambda
_{\alpha _{i}}\pi \left( b_{\alpha _{i}}\right) \in W_{i}\pi \left(
V_{1}\right) $ and $\lambda _{i}^{-1}\lambda _{\alpha _{i}}\in W_{i_{1}}$.
Thus there exist $v_{i}\in V_{1}$ such that $\lambda _{\alpha _{i}}b_{\alpha
_{i}}-\lambda _{i}v_{i}\in M$. For every $i\geq i_{1}$, we can write%
\begin{equation*}
\lambda _{\alpha _{i}}b_{\alpha _{i}}-\lambda _{i}v_{i}=\lambda _{i}\left(
\lambda _{i}^{-1}\lambda _{\alpha _{i}}b_{\alpha _{i}}-v_{i}\right)
\end{equation*}%
where%
\begin{equation*}
\lambda _{i}^{-1}\lambda _{\alpha _{i}}b_{\alpha _{i}}-v_{i}\in \left(
W_{i_{1}}B+B\right) \cap M\subseteq A\cap M\text{.}
\end{equation*}%
Since $A\cap M$ is bounded in $M$, there exists $i_{2}\geq i_{1}$ such that%
\begin{equation*}
\lambda _{i_{2}}^{-1}\lambda _{\alpha _{i_{2}}}b_{\alpha
_{i_{2}}}-v_{i_{2}}=\lambda _{i_{2}}(\lambda _{i_{2}}^{-1}\lambda _{\alpha
_{i_{2}}}b_{\alpha _{i_{2}}}-v_{i_{2}})\in V_{1}
\end{equation*}%
and hence 
\begin{equation*}
\lambda _{\alpha _{i_{2}}}b_{\alpha _{i_{2}}}=(\lambda _{\alpha
_{i_{2}}}b_{\alpha _{i_{2}}}-\lambda _{i_{2}}v_{i_{2}})+\lambda
_{i_{2}}v_{i_{2}}\in V_{1}+W_{0}V_{1}\subseteq V_{0}\text{.}
\end{equation*}%
This contradicts the fact that $\lambda _{\alpha _{i_{2}}}b_{\alpha
_{i_{2}}} $ does not belong to $V_{0}$.

\item This is a particular instance of (1), since for a separable
non-Archimedean valued field $R$, the separable Fr\'{e}chet spaces over $R$
are precisely the non-Archimedean Polish vector spaces over $R$---see \cite[%
Definition 3.5.1, Theorem 3.5.2]{perez-garcia_locally_2010}.

\item This follows from (10) and (11), since for a separable non-Archimedean
valued field $R$, the separable Banach spaces over $R$ are precisely the
locally bounded separable Fr\'{e}chet spaces over $R$; see \cite[Definition
2.1.1, Theorem 3.6.2]{perez-garcia_locally_2010}.

\item Suppose that $G$ is a locally compact Polish group, and $H\subseteq G$
is a closed subgroup. If $G$ is a Lie group, then $H$ is a Lie group by the
Closed Subgroup Theorem for Lie groups \cite[Theorem 20.12]%
{lee_introduction_2013}. If $H$ is normal in $G$, then $G/H$ is a Lie group
by \cite[Theo7]{lee_introduction_2013}. If $G$ is abelian and both $H$ and $%
G/H$ are Lie groups, then $G$ is a Lie group by \cite{tao_central_2011}; see
also \cite{gleason_spaces_1950,kuranishi_euclidean_1950}. Notice also that a
continuous group homomorphism between real Lie groups is real analytic; see 
\cite[Theorem 2.11.2]{varadarajan_lie_1984}.

This assertion also follows from \cite[Theorem 2.6(1)]%
{moskowitz_homological_1967}, after noticing that a locally compact abelian
Polish group is a Lie group if and only if it \emph{has no small subgroups},
which means that it has a zero neighborhood that does not contain any
nontrivial subgroups---see \cite{moskowitz_homological_1967}.

\item This follows from (1) and (2), since a locally compact Polish group is
totally disconnected if and only if it is non-Archimedean.

\item Compactly generated abelian Polish groups form a thick subcategory of $%
\mathcal{A}_{\mathbb{Z}}$ by \cite[Theorem 2.6(2)]%
{moskowitz_homological_1967}.

\item Locally compact abelian Polish topological torsion groups form a thick
subcategory of $\mathcal{A}_{\mathbb{Z}}$ by \cite[3.17]%
{armacost_structure_1981}.

\item That locally compact abelian Polish topological $p$-groups form a
thick subcategory of $\mathcal{A}_{\mathbb{Z}}$ follows easily from (10)
after observing that a topological torsion locally compact abelian Polish
group $G$ is a topological $p$-group if and only if it is equal to its $%
\mathbb{Z}_{p}$-component in the sense of \cite[Definition 4.12]%
{armacost_structure_1981}; see \cite[Remark 3.9(a) and Example 4.13(a)]%
{armacost_structure_1981}.

\item Locally compact Polish abelian groups that have finite ranks form a
thick subcategory of $\mathcal{A}_{\mathbb{Z}}$ by \cite[Proposition 2.9]%
{hoffmann_homological_2007}.
\end{enumerate}
\end{proof}

A Borel function $f:X\rightarrow Y$ between locally compact Polish spaces is
called \emph{locally bounded }if, for every compact subset $C$ of $X$, $%
f\left( C\right) $ has compact closure in $Y$.

\begin{corollary}
\label{Corollary:locally-bounded}Suppose that $\varphi :X\rightarrow Y$ is a 
$R$-homomorphism between $R$-modules with a Polish cover $X=\hat{X}/N$ and $%
Y=\hat{Y}/M$, where $\hat{X}$ and $M$ are locally compact. Then the
following assertions are equivalent:

\begin{enumerate}
\item $\varphi $ is Borel-definable;

\item $\varphi $ has a \emph{locally bounded }Borel lift $\hat{X}\rightarrow 
\hat{Y}$.
\end{enumerate}

If furthermore $\hat{X}$ is finite-dimensional, then these conditions are
equivalent to:

\begin{enumerate}
\item $\varphi $ has an\emph{\ approximately additive locally continuous }%
Borel lift.
\end{enumerate}
\end{corollary}

\begin{proof}
Suppose that (1) holds. By Theorem \ref{Theorem:left-heart-B2}, we can write 
$\varphi =\psi \circ \sigma ^{-1}$ where $Z=\hat{Z}/L$ is an $R$-module with
a Polish cover, $\psi :Z\rightarrow Y$ and $\sigma :Z\rightarrow X$ are
liftable $R$-homomorphisms such that $\hat{Z}$ is an extension of $\hat{X}$
by $M$ (and, in particular, a locally compact Polish $R$-module) and $\sigma 
$ lifts to a surjective continuous $R$-homomorphism $\hat{\sigma}:\hat{Z}%
\rightarrow \hat{X}$. By the main theorem in \cite{kehlet_cross_1984}, we
have that $\hat{\sigma}$ has a Borel locally bounded right inverse $g:\hat{X}%
\rightarrow \hat{Z}$. (Notice that \cite{kehlet_cross_1984} adopts the
terminology of \cite[Section 51]{halmos_measure_1950} where the
\textquotedblleft Baire $\sigma $-algebra\textquotedblright\ on a
topological space is the $\sigma $-algebra generated by the compact $%
G_{\delta }$ sets. This coincides with the Borel $\sigma $-algebra in the
case of locally compact Polish spaces. Thus, a \textquotedblleft Baire
function\textquotedblright\ between locally compact Polish spaces in the
sense of \cite{kehlet_cross_1984} is just a Borel function.) Thus, if $\hat{%
\psi}:\hat{Z}\rightarrow \hat{Y}$ is a continuous $R$-homomorphism that
lifts $\psi $, then we have that $\hat{\psi}\circ g$ is a locally bounded
Borel lift for $\varphi $. If furthermore $\hat{X}$ is finite-dimensional,
then by \cite[Theorem 8]{mostert_sections_1956}, we have that there exist a
zero neighborhood $V$ of $\hat{X}$ and a continuous function $g:V\rightarrow 
\hat{Z}$ that is a right inverse for $\hat{\sigma}|_{\hat{\sigma}^{-1}\left(
V\right) }:\hat{\sigma}^{-1}\left( V\right) \rightarrow V$. Thus, if $\hat{%
\psi}:\hat{Z}\rightarrow \hat{Y}$ is a continuous $R$-homomorphism that
lifts $\psi $, then we have that $\hat{\varphi}:=\hat{\psi}\circ g$ is a
continuous local lift for $\varphi $. We can extend $\hat{\varphi}$ to a
Borel lift on $\hat{X}$ by Lemma \ref{Lemma:locally-continuously}. Since $M$
is locally compact, we have that $M$ has a basis of zero neighborhoods that
are compact, and in particular closed in $\hat{Y}$. Therefore we have that $%
\varphi $ has an approximately additive Borel lift that is continuous in a
zero neighborhood of $\hat{X}$ by Proposition \ref%
{Proposition:approximately-additive}.
\end{proof}

\begin{corollary}
\label{Corollary:locally-analytic}Suppose that $\varphi :G\rightarrow H$ is
a group homomorphism between abelian groups with a real Lie cover $G=\hat{G}%
/N$ and $H=\hat{H}/M$. Then the following assertions are equivalent:

\begin{enumerate}
\item $\varphi $ is Borel-definable;

\item $\varphi $ has an approximately additive Borel lift that is \emph{%
analytic} on an open zero neighborhood in $\hat{G}$.
\end{enumerate}
\end{corollary}

\begin{proof}
The proof is the same as the proof of Corollary \ref%
{Corollary:locally-bounded} the fact that a continuous group homomorphism
between real Lie groups is real analytic \cite[Theorem 2.11.2]%
{varadarajan_lie_1984}, and that if $\pi :X\rightarrow Y$ is a surjective
continuous homomorphism between real Lie groups, then there exists a zero
neighborhood $V$ of $X$ and an analytic right inverse $g:V\rightarrow X$ for 
$\pi |_{\pi ^{-1}\left( V\right) }:\pi ^{-1}\left( V\right) \rightarrow V$ 
\cite[Theorem 2.9.5]{varadarajan_lie_1984}; see also Remark \ref%
{Remark:approximately-additive-analytic}.
\end{proof}

%    Bibliography styles amsplain or author-year (using natbib) are
%    also acceptable.
\bibliographystyle{amsalpha}
\bibliography{bibliography}
%    See note above about multiple indexes.
\printindex

\end{document}